\documentclass[12pt,a4paper]{article}

\usepackage[utf8]{inputenc}
\usepackage[T1]{fontenc}
\usepackage{graphicx}
\usepackage{epstopdf}
\usepackage[margin=20mm]{geometry} 
\usepackage{amssymb}
\usepackage{latexsym}
\usepackage{listings}
\usepackage{amsfonts}
\usepackage{amsthm}
\usepackage{bm}
\usepackage{amsmath}
\usepackage{thmtools}
\usepackage[linesnumbered]{algorithm2e}
\usepackage{biblatex}
\graphicspath{ {Figures/} } 
\usepackage{color}
\usepackage{standalone} 
\usepackage[table,x11names]{xcolor}
\usepackage{soul}
\usepackage{multirow}
\usepackage{multicol}			
\usepackage{caption}			
\usepackage{subcaption}		
\usepackage[colorlinks]{hyperref}
\hypersetup{citecolor = {blue},pdfauthor=author}
\usepackage{cleveref}
\usepackage{tikz}
\usepackage{tikz-cd}
\usepackage{pgfplots}
\usepackage{pgfplotstable}
\usepackage{sectsty}			
\usepackage{physics}
\usepackage{array}
\usepackage{float}
\usepackage{forest}
\usepackage{xparse}
\usepackage{lipsum}
\usepackage{multicol}
\usepackage{todonotes}
\usepackage{csvsimple}

\usetikzlibrary{shapes,arrows,arrows.meta,decorations.pathmorphing,decorations.markings,decorations.pathreplacing,patterns,shapes.geometric}
\usepgfplotslibrary{groupplots}
\pgfplotsset{compat=1.13}
\tikzcdset{scale cd/.style={every label/.append style={scale=#1},
    cells={nodes={scale=#1}}}}

\newcommand{\numberthis}{\refstepcounter{equation}\tag{\theequation}}
\newcommand{\R}{\mathbb{R}}

\newcommand{\half}{\frac{1}{2}}

\DeclareMathOperator{\arcsinh}{asinh}
\pgfkeys{/pgf/declare function={arcsinh(\x) = ln(\x + sqrt(\x^2+1));}}

\newcolumntype{C}{>{$}c<{$}}

\newtheorem{theorem}{Theorem}[section]
\theoremstyle{definition}

\declaretheoremstyle[
  numbered=yes,
  numberlike=theorem,
  spaceabove=1em plus 0.75em minus 0.25em,
  spacebelow=1em plus 0.75em minus 0.25em,
  qed={}
]{exmpstyle}

\newtheorem{lemma}[theorem]{Lemma}
\newtheorem*{remark}{Remark}

\newcommand{\logLogSlopeTriangleText}[6]
{

    \pgfplotsextra
    {
        \pgfkeysgetvalue{/pgfplots/xmin}{\xmin}
        \pgfkeysgetvalue{/pgfplots/xmax}{\xmax}
        \pgfkeysgetvalue{/pgfplots/ymin}{\ymin}
        \pgfkeysgetvalue{/pgfplots/ymax}{\ymax}

        \pgfmathsetmacro{\xArel}{#1}
        \pgfmathsetmacro{\yArel}{#3}
        \pgfmathsetmacro{\xBrel}{#1-#2}
        \pgfmathsetmacro{\yBrel}{\yArel}
        \pgfmathsetmacro{\xCrel}{\xArel}

        \pgfmathsetmacro{\lnxB}{\xmin*(1-(#1-#2))+\xmax*(#1-#2)} 
        \pgfmathsetmacro{\lnxA}{\xmin*(1-#1)+\xmax*#1} 
        \pgfmathsetmacro{\lnyA}{\ymin*(1-#3)+\ymax*#3} 
        \pgfmathsetmacro{\lnyC}{\lnyA+(-#4)*(\lnxA-\lnxB)}
        \pgfmathsetmacro{\yCrel}{\lnyC-\ymin)/(\ymax-\ymin)} 

        \coordinate (A) at (rel axis cs:\xArel,\yArel);
        \coordinate (B) at (rel axis cs:\xBrel,\yBrel);
        \coordinate (C) at (rel axis cs:\xCrel,\yCrel);

        \draw[#5]   (A) node[pos=0.5,anchor=north] {1}
                    (B)-- 
                    (C) node[xshift=0.2,yshift=0.2,pos=0.6,above,sloped,#6] {slope #4};
    }
}

\newcommand{\logLogSlopeTriangle}[5]{\logLogSlopeTriangleText{#1}{#2}{#3}{#4}{#5}{}}

\usepackage{CJKutf8}
\newcommand\keywords[1]{%
    \begingroup
    \let\and\\
    \par
    \noindent\textit{Keywords:} #1\par
    \endgroup
}

\newcommand{\setfoot}[2]{%
    \footnote{#2}%
    \newcounter{#1}%
    \setcounter{#1}{\value{footnote}}%
}

\newcommand{\getfoot}[1]{%
    \footnotemark[\value{#1}]%
}

\pgfplotsset{
        compat=1.13,
        my style/.style={
            width=0.5\textwidth,height=0.5\textwidth,
            xlabel=$n$,
            ylabel=$\log(rel\; err)$,
        },
        my other style/.style={
            width=0.5\textwidth,height=0.3\textwidth,
            xlabel=$n$,
        },
        my legend style/.style={
            legend entries={
                2nd order,
                4th order,
            },
            legend style={
                at={([yshift=2pt]1.0,1.15)},
                anchor=north,
            },
            legend columns=2,
        },
        parallel legend style/.style={
            legend entries={$\kappa_\parallel=1.0$,
                $\kappa_\parallel=10^3$,
                $\kappa_\parallel=10^6$,
                $\kappa_\parallel=10^9$
            },
            legend style={
                at={([yshift=2pt]1.1,1.2)},
                anchor=north,
            },
            legend columns=4,
        },
        NIMROD parallel legend style/.style={
            legend entries={$\kappa_\parallel=1$,
                $\kappa_\parallel=10^3$,
                $\kappa_\parallel=10^5$,
                $\kappa_\parallel=10^6$,
                $\kappa_\parallel=10^7$,
                $\kappa_\parallel=10^9$,
                $\kappa_\parallel=10^{10}$,
            },
            legend style={
                at={([yshift=10pt]1.1,1.225)},
                anchor=north,
            },
            legend columns=4,
        },
        cycle multiindex* list={
            blue!75!black,
            red!75!black,
            green!75!black,
            cyan!75!black,
            magenta!75!black,
            yellow!75!black,
            black!75!black,
                \nextlist
            mark=*,
            mark=square*,
            mark=x,
                \nextlist
        }}

\title{A provably stable numerical method for the anisotropic diffusion equation in confined magnetic fields.}
\author{
    D. Muir\setfoot{cor}{Corresponding author} \setfoot{anu}{Mathematical Sciences Institute, Australian National University, Australian Capital Territory~2601, \textsc{Australia}.},
    K. Duru\getfoot{anu},
    M. Hole\getfoot{anu},
    S. Hudson\setfoot{pppl}{Princeton Plasma Physics Laboratory, New Jersey}}

\date{April 2024}

\begin{document}

    \maketitle
    
    \begin{abstract}
        We present a novel numerical method for solving the anisotropic diffusion equation in magnetic fields confined to a periodic box which is accurate and provably stable.
        We derive energy estimates of the solution of the continuous initial boundary value problem. A discrete formulation is presented using operator splitting in time with the summation by parts finite difference approximation of spatial derivatives for the perpendicular diffusion operator. Weak penalty procedures are derived for implementing both boundary conditions and parallel diffusion operator obtained by field line tracing.
        We prove that the fully-discrete approximation is unconditionally stable. Discrete energy estimates are shown to match the continuous energy estimate given the correct choice of penalty parameters. 
        A nonlinear penalty parameter is shown to provide an effective method for tuning the parallel diffusion penalty and significantly minimises rounding errors.
        Several numerical experiments, using manufactured solutions, the ``NIMROD benchmark'' problem and a single island problem, are presented to verify numerical accuracy, convergence, and asymptotic preserving properties of the method.
        Finally, we present a magnetic field with chaotic regions and islands and show the contours of the anisotropic diffusion equation reproduce key features in the field.
    \end{abstract}

\section{Introduction}

Magnetic confinement fusion devices are defined by extremely strong toroidal magnetic fields, which result in the transport processes along magnetic field lines being orders of magnitude faster than those perpendicular to the field, with the ratio of diffusion coefficients $\kappa_\parallel/\kappa_\perp$ exceeding $10^{10}$~\cite{fitzpatrick_helical_1995,gunter_modelling_2005,hudson_temperature_2008}. Here, $\kappa_\parallel$ is the diffusion coefficient in the direction parallel to the magnetic field and $\kappa_\perp$ the coefficient perpendicular to the field.
The anisotropic diffusion equation presents a simplified model for heat transport in confinement fusion devices. The measurement of which is crucial when considering the impacts of heat deposition on plasma facing components.

The separation in diffusive scales results in a numerically challenging multiscale problem, which often manifests as an ill-conditioned linear algebraic problem in numerical computations.
In order to avoid numerical errors overwhelming the small scale diffusion, a common approach in similar problems is to lay out a coordinate system which allows one to easily separate the scales. However, when a magnetic field is chaotic or in the presence of islands this may not be possible.

\citeauthor{hudson_temperature_2008}~\cite{hudson_temperature_2008} showed that the contours of the solution to the time independent anisotropic diffusion equation in a magnetic field closely resembles the features of the underlying magnetic field. 
The contours may therefore provide an alternative to flux coordinates upon which we may base discretisations in instances where straight field line or flux coordinates are unavailable.
More recently, authors have used the solution to the diffusion equation as a proxy for measuring the integrability of the magnetic field \cite{helander_heat_2022,paul_heat_2022,drivas_distribution_2022}.
The interest in equilibrium features has typically led to use of time independent methods~\cite{hudson_temperature_2008,helander_heat_2022,paul_heat_2022}. However, time dependant methods have also been developed \cite{mentrelli_asymptotic-preserving_2012,chacon_asymptotic-preserving_2014}.

Many of the time dependant methods aim to be asymptotic preserving.
A numerical scheme has this property if it is consistent with the solution to the limit problem $\kappa_\perp/\kappa_\parallel \to 0$ for all stable discretisation parameters \cite{degond_asymptotic-preserving_2013}. See \citeauthor{mentrelli_asymptotic-preserving_2012}~\cite{mentrelli_asymptotic-preserving_2012} and  \citeauthor{chacon_asymptotic-preserving_2014}~\cite{chacon_asymptotic-preserving_2014} for more elaborate discussions on asymptotic preserving schemes for the diffusion equation.

Both finite difference \cite{gunter_modelling_2005,gunter_mixed_2009,van_es_finite-difference_2014,chacon_asymptotic-preserving_2014} and finite element \cite{gunter_finite_2007,mentrelli_asymptotic-preserving_2012} schemes have been utilised to compute the perpendicular diffusion. 
For the parallel diffusion, an established method of approximation is to use field line tracing either for the purposes of approximation by finite differences \cite{hudson_temperature_2008} or for integrating the solution along them \cite{chacon_asymptotic-preserving_2014}. 
Other methods for the fully 3D problem include using a Fourier representation of the field in the toroidal direction \cite{gunter_modelling_2005,gunter_finite_2007,paul_heat_2022}.
Our method follows the common ``finite difference with field line tracing'' approach, but unlike previous approaches we approximate the perpendicular operator with a summation by parts finite difference scheme and the parallel diffusion with a novel penalty method, with the goal of deriving a provable stable numerical method.

In this paper we present such a numerical method for solving the field aligned anisotropic diffusion equation, with rigorous mathematical support. We derive energy estimates of the solution of the underlying initial boundary value problem (IBVP). In the perpendicular direction we approximate the spatial derivatives using summation-by-parts (SBP) finite difference operators \cite{mattsson_summation_2004,mattsson_summation_2012,nordstrom_summation-by-parts_2016}. Boundary conditions are implemented weakly using the simultaneous approximation term (SAT) approach.
For the parallel operator we implement a novel penalty method, which is based on the well established field line tracing approach to construct the parallel transport, but combines it with a penalty which is similar to the SAT concept.

We prove numerical stability for the semi-discrete approximation by deriving discrete energy estimates mimicking the continuous energy estimates. We note that initial ideas were presented in \cite{muir_efficient_2023} for a simple one dimensional model. In the present work we extend the ideas to multiple dimensions and magnetic fields.
The numerical method is implemented in a \textit{Julia} code \textit{FaADE.jl} available on \textit{GitHub}\footnote{https://github.com/Spiffmeister/FaADE.jl}.

The time-derivative for the isotropic diffusion is approximated using a general theta method, where the fully-discrete algebraic problem is solved by an operator splitting technique which accurately separates the disparate scales in the solution. The parallel penalty term is approximated using an implicit midpoint rule.
We prove that the fully-discrete problem is unconditionally stable.
Numerical experiments are performed showing convergence for the perpendicular solution. We also use the ``NIMROD benchmark'' \cite{sovinec_nonlinear_2004} to show convergence for the full scheme even in the case where the parallel diffusion is large and demonstrate the method is asymptotic preserving. 
A single island self convergence test is also performed.
We present an example problem magnetic field with chaotic regions and islands and show the contours of the anisotropic diffusion equation reproduce key features in the field.

This paper is organised as follows;
the following section, \S\ref{sec:Anisotropic Diffusion Equation}, details the continuous problem and the derivation of the field aligned form. We present the theorem for well-posedness of the field aligned form which is used to constrain properties of the numerical scheme. We also provide a description of the field line tracing approach later used in the parallel map.
In \S\ref{sec:Numerical approach} we discuss the numerical approach, beginning with the description of the summation by parts operators. Following this we describe the semi-discrete problem in \S\ref{sub:sec:Discretisation} and then the discrete form of the parallel map. Stability is given for the semi-discrete case. The numerical description concludes in \S\ref{sec:sub:fully-discrete approximation} with a description of the fully discrete approximation and a proof for the unconditional stability of the method.
We follow this with numerical results in \S\ref{sec:Numerical Results}. We show convergence of the perpendicular diffusion by the method of manufactured solutions, of the full scheme by the ``NIMROD benchmark'' and a self convergence test with a single island. We finalise the numerical section with a single case showing qualitatively that our code generates contours which match qualitative features of the magnetic field as we expect.
We conclude in \S\ref{sec:Conclusion}.

\section{Anisotropic Diffusion Equation}\label{sec:Anisotropic Diffusion Equation}
Let $u:\Omega\cross[0, T]\to\R$ be a scalar field denoting temperature with $\Omega\subset\R^3$. The anisotropic diffusion equation is,
\begin{align}\label{eq:ADE}
    \pdv{u}{t}  &= \nabla\cdot(\bm{K}\nabla u) + F,
\end{align}
where $\bm{K}\in\R^{3\times 3}_{\geq0}$ is the diffusion tensor which is real, symmetric positive definite, and $F:\Omega\cross[0, T]\to\R$ is the source term. 
We introduce the magnetic field $\bm{B}:\R^3\to\R^3$.
Since the transport occurs predominantly in the direction of the magnetic field in fusion plasmas, the common approach to solve the anisotropic diffusion equation is to transform the anisotropic diffusion equation  \eqref{eq:ADE} into coordinates that are aligned with the magnetic field $\bm{B}$ \cite{gunter_modelling_2005,hudson_temperature_2008,chacon_asymptotic-preserving_2014}.
Letting $\kappa_\parallel$ and $\kappa_\perp$ be real positive numbers, with $\kappa_\parallel \gg \kappa_\perp$, where the $\kappa_\parallel$ is the  diffusion coefficient in the direction parallel to magnetic field $\bm{B}$ and $\kappa_\perp$ is the  diffusion coefficient in the direction perpendicular to magnetic field $\bm{B}$.   The diffusion tensor is
\begin{align}\label{eq:diff-tensor}
    \bm{K} = \kappa_\perp\bm{I} + (\kappa_\parallel - \kappa_\perp)\frac{\bm{B}\bm{B}^T}{\|\bm{B}\|^2}.
\end{align}
We define the gradient operator along the direction parallel to the magnetic field $\bm{B}$ 
\begin{align}\label{eq:gradient splitting}
    \nabla_\parallel = \bm{B}(\bm{B}\cdot\nabla)/\|\bm{B}\|^2.
\end{align}
Using \eqref{eq:diff-tensor} and \eqref{eq:gradient splitting} in \eqref{eq:ADE}, we can derive the following compact form of field aligned form of the anisotropic diffusion equation,
\begin{align}\label{eq:ADE field aligned}
    \pdv{u}{t}  &= \nabla\cdot(\kappa_\perp\nabla u) + \nabla\cdot\left(\hat{\kappa}_\parallel\nabla_\parallel u\right) + F, \quad \hat{\kappa}_\parallel = {\kappa}_\parallel-{\kappa}_\perp.
\end{align}
The small scale perpendicular diffusion is then considered to be isotropic, with a strongly anisotropic component along the magnetic field. Numerically this can be used to treat the problem in an operator-split form.

For the moment we  assume that $\bm{B}\cdot\bm{n}=0$ where $\bm{n}$ is the outward pointing normal vector of the domain boundaries. 
This ensures all field lines are confined to the domain.

We will consider problem \eqref{eq:ADE field aligned} with the smooth initial condition,
\begin{align}\label{eq:ADE initial condition}
    u(x,y,z,0) = f(x,y,z),
\end{align}
which is well-defined in $(x,y,z) \in \Omega$.
The most relevant set of boundary conditions for our current analysis are Dirichlet conditions in the $x$-direction,
\begin{align}\label{eq:Boundary x Dirichlet}
    u(x_L,y,z,t) = g_L(y,z,t) \quad\text{and}\quad u(x_R,y,z,t)=g_R(y,z,t),
\end{align}
and periodic conditions in the $z$ and $y$-directions,
\begin{subequations}\label{eq:Boundary y periodic}
\begin{align}
    u(x,y_L,z,t) = u(x,y_R,z,t) \quad\text{and}\quad
    \left.\pdv{u}{y}\right|_{y_L} = \left.\pdv{u}{y}\right|_{y_R}, \\
    u(x,y,z_L,t) = u(x,y,z_R,t) \quad\text{and}\quad
    \left.\pdv{u}{z}\right|_{z_L} = \left.\pdv{u}{z}\right|_{y_R}.
\end{align}
\end{subequations}
Here $g_L(y,z,t)$ and $g_R(y,z,t)$ are boundary data which are compatible with the initial condition, that is $g_L(y,z,0)=f(x_L,y,z)$ and $g_R(y,z,0)=f(x_R,y,z)$.
Topologically this set of boundary conditions is equivalent to solutions in a hollow torus. We note however that the analysis can be extended to any well-posed  boundary conditions for \eqref{eq:ADE field aligned}. 

The anisotropic diffusion equation has a well known energy bound.
\begin{theorem}\label{theo:well-posedness}
    Consider the field aligned anisotropic diffusion equation \eqref{eq:ADE field aligned} with compactly supported  initial condition $f$ given in \eqref{eq:ADE initial condition} and boundary conditions given by \eqref{eq:Boundary x Dirichlet}--\eqref{eq:Boundary y periodic}. Let $\|u\|^2 = \int_\Omega u^2 \dd x\dd y\dd z$, if $F =0$, $g_L(y,z,t) = g_R(y,z,t)=0$ then,
    \begin{align}\label{eq:continuous_energy_estimate}
        \|u\| \leq \|f\| , \quad \forall t \ge 0 .
    \end{align}
\end{theorem}
A stable numerical method will mimic the energy estimate given by \eqref{eq:continuous_energy_estimate}. Theorem \ref{theo:well-posedness} is a well known result and can be found in many textbooks \cite{evans_partial_2010}. However,  we have included a sketch of the proof of the theorem in the appendix which we will closely emulate in our numerical analysis.

\begin{remark}
    The present boundary conditions are suitable for analysing the shape of field features in hollow tori. In this paper we present examples in a periodic box, which is topologically equivalent to a hollow torus.
\end{remark}

To facilitate the discussion in the coming sections, we introduce the notation for the  field aligned anisotropic diffusion equation \eqref{eq:ADE field aligned},
\begin{align}\label{eq:ADE field aligned split}
    \pdv{u}{t}  &=  \mathcal{P}_\perp u +  \mathcal{P}_\parallel u + F,
\end{align}
where
\begin{align}\label{eq:split-operators}
\mathcal{P}_\perp u=\nabla\cdot(\kappa_\perp\nabla u), \quad 
\mathcal{P}_\parallel u = \nabla\cdot\left(\hat{\kappa}_\parallel\nabla_\parallel u\right).
\end{align}
Here $\mathcal{P}_\perp$ denotes the isotropic perpendicular diffusion operator and $\mathcal{P}_\parallel$ is the anisotropic parallel diffusion operator. Our goal is to efficiently approximate the operators $\mathcal{P}_\perp$ and $\mathcal{P}_\parallel$ using different numerical methods and combine them in a stable manner with solid mathematical support.
In particular, we will approximate the perpendicular diffusion operator $\mathcal{P}_\perp$  using SBP finite difference operators and implement boundary conditions  weakly using SAT.
The parallel diffusion operator $\mathcal{P}_\parallel$ is approximated using field line tracing approach and a novel parallel nonlinear penalty method for volume corrections.

Going forward we will ignore the hat and write $\kappa_\parallel$ to simplify notation.

\section{The numerical approach}\label{sec:Numerical approach}

In this section we present the numerical method and prove numerical stability.
We approximate the solution on a single $z=const$ plane, setting derivatives in this direction to zero, i.e. so $\partial/\partial z \equiv 0$. 
On this computational plane all derivatives for the perpendicular diffusion operator $\mathcal{P}_\perp$  will be approximated by SBP finite difference operators and boundary conditions implemented weakly using penalties. 
The parallel diffusion operator $\mathcal{P}_\parallel$ is modelled through field line tracing and parallel numerical solution  enforced weakly using penalty in the plane. For the semi-discrete approximation  we will show that the norm of the numerical solution is bounded by the norm of the initial condition for all times. Numerical experiments will verify the analysis.
This approximation is equivalent to assuming that either the magnetic field is perpendicular to the computational plane, or that the derivative in the $z$-direction is so negligible as to have no effect, which may be the case in an axisymmetric system.
Numerically approximating the $z$-derivatives would require a fully three dimensional simulation. The technique we introduce in this paper may be extended to three dimensions at a later stage.

\subsection{SBP operators}\label{sub:sec:sbp_operators}
We first introduce the SBP finite difference operators \cite{mattsson_summation_2004, mattsson_summation_2012} used in this work.
To begin, we consider the 1D spatial interval $x \in [x_L, x_R]$ and discretise it into $n$ grid points with a uniform spatial step $\Delta{x} >0$, having
\begin{align}\label{eq:1D_grid}
   x_j&=x_L + (j-1)\Delta x, \quad \Delta x=\frac{x_R-x_L}{n-1},  \quad j = 1, 2, \cdots n.
\end{align}
Let the vector $\mathbf{u} = [u_1(t), u_2(t), \cdots u_n(t)]^T\in\R^n$ denote the semi-discrete scalar field on the grid. Furthermore, let $D_x, D_{xx}^{(\kappa)} \in \mathbb{R}^{n\times n}$ denote discrete approximations of the first and second spatial derivatives on the grid \eqref{eq:1D_grid},  that is $(D_x\mathbf{u})_j \approx \partial u/\partial x|_{x=x_j}$ and $(D_{xx}^{(\kappa)}\mathbf{u})_j \approx \partial\left(\kappa\partial u/\partial x\right)/\partial x|_{x=x_j}$, where $\kappa >0$ is the diffusion coefficient.
Let $H=H^T>0$ be a matrix which defines a discrete norm and whose entries are the positive weights of a composite quadrature rule. 
The discrete operators $D_x, D_{xx}^{(\kappa)}$ are called SBP operators if
\begin{align}\label{eq:sbp_x}
    & D_x=H^{-1}Q,  \quad  Q+Q^T= B:=\emph{diag}([-1,0,\cdots,1]), \\
    &H=H^T, \quad \mathbf{u}^T H \mathbf{u} > 0, \quad \forall \mathbf{u}\in\R^n,\\
\label{eq:sbp_xx}
    & D_{xx}^{(\kappa)}  = H^{-1}(-M^{(\kappa)} + BKD_x),  \quad M^{(\kappa)} = (M^{(\kappa)})^T, \quad\mathbf{u}^TM^{(\kappa)} \mathbf{u}\geq 0,
\end{align}
where 
$K=\operatorname{diag}([\kappa(x_1),\kappa(x_2),\cdots,\kappa(x_n)])$ is a diagonal matrix of the diffusion coefficients.

Note that by construction, $B$ picks out the first and last elements of a vector such that $B\bm{u} = [-\bm{u}_1, 0,\cdots, 0, \bm{u}_N]$. Therefore, the first derivative SBP operator can be seen to mimic integration by parts discretely,
\begin{align}
    \int_\Omega u \pdv{u}{x} \dd x &= \pdv{u}{x} u - u(x_L)^2 + u(x_R)^2, \\
    (\bm{u},H D_x \bm{u}) &= \bm{u} Q \bm{u} = \bm{u} Q^T \bm{u} - u_1^2 + u_N^2.
\end{align}
Furthermore, for $\kappa=1$ the second derivative operator gives,
\begin{align}
    \int_\Omega u \pdv[2]{u}{x} \dd x &= -\left(\pdv{u}{x}\right) - \left.\pdv{u}{x}u\right|_{x_L} + \left.\pdv{u}{x}u\right|_{x_R}, \\
    \bm{u}^T H D_{xx}^{(\kappa)} \bm{u} &= -\bm{u}^{T} M \bm{u} - \bm{u}_1[D_x\bm{u}]_1 + \bm{u}_N[D_x\bm{u}]_N,
\end{align}
where $[D_x\bm{u}]_i$ is the $i$th component of the vector $D_x\bm{u}$.

The SBP operators $D_x$ and  $D_{xx}^{(\kappa)}$ are called \textit{fully compatible} \cite{duru_stable_2014} if
\begin{align}\label{defn:eq:fully compatible SBP operator}
    M^{(\kappa)} = D_x^T \left(KH\right) D_x + R_x^{(\kappa)},  \quad R^{(\kappa)} = (R^{(\kappa)})^T, \quad\mathbf{u}^TR^{(\kappa)} \mathbf{u}\geq 0.
\end{align}
Here, $R_x^{(\kappa)}$ is a remainder operator and represents the truncation error for the SBP operator~\cite{mattsson_summation_2012}.
We will use fully compatible and diagonal norm SBP operators with 
$H=\Delta{x}\operatorname{diag}([h_1,h_2,\cdots,h_n])$, where $h_j >0$. Coefficients for 2nd order accurate SBP operators are given in the appendix \ref{sec:appendix:coeffs sbp operators}. For a more elaborate list of SBP operator coefficients we refer readers to \cite{mattsson_summation_2012}.
The SBP properties \eqref{eq:sbp_x}--\eqref{eq:sbp_xx} will be useful in proving numerical stability. 
In this study we will consider constant perpendicular diffusion coefficient $\kappa_\perp >0$. However, by using the variable coefficient SBP operators for second derivatives \cite{mattsson_summation_2012} we can  extend the analysis to variable diffusion coefficients and variable grid spacing. 

\begin{remark}\label{rmk:stencil order}
    Fully compatible SBP operators with $(2p)$-th order interior stencils have $(p-1)$-th stencils on the boundary with $p\ge 1$. Convergent numerical approximations with fully compatible SBP operators will yield $(p+1)$-th order global convergence rate \cite{svard_order_2006}. Here we will use fully compatible SBP operators with $2$nd and $4$th interior stencils. These approximations will converge at $2$nd and $3$rd order global convergence rates.
\end{remark}

The 1D SBP operators can be extended to multiple dimensions using the Kronecker product $\otimes$.
Consider the discretisation of the above 2D domain $(x,y) \in \Omega =[x_L,x_R]\times[y_L,y_R]$, where $x_L<x_R$, $y_L<y_R\in\R$. The discretisation has uniform spatial steps $\Delta x>0$, $\Delta y>0$ in $x$ and $y$ so that,
\begin{align}\label{eq:CartesianGrid}
   \begin{split}
    x_i&=x_L + (i-1)\Delta x, \quad \Delta x=\frac{x_R-x_L}{n_x-1}, \quad i=1,2,\dots,n_x,  \\
    y_j&=y_L + (j-1)\Delta y, \quad \Delta y=\frac{y_R-y_L}{n_y-1}, \quad j=1,2,\dots,n_y.
\end{split} 
\end{align}
Let $u(x_i, y_j, t)\approx u_{ij}(t)$, denote the numerical approximation of the solution on the grid. We rearrange the grid function $u_{ij}(t)$ row-wise as a column vector using $\bm{u}=[u_{1,1},u_{2,1},\dots,u_{n_x,n_y}]^T$, and introduce the vector $\bm{u}(t)\in\R^{(n_xn_y)}$ which denotes the semi-discrete scalar field on the grid. 

We extend the 1D operators to 2D using Kronecker products,
\begin{align}
    \mathbf{D}_{x} = D_x\otimes I_{n_y}, \quad  \mathbf{D}_{xx}^{(\kappa_\perp)} = D_{xx}^{(\kappa_\perp)}\otimes I_{n_y}, \quad \mathbf{H}_{x} = H_x\otimes I_{n_y}, \\
     \mathbf{D}_{y} = I_{n_x}\otimes D_y, \quad  \mathbf{D}_{yy}^{(\kappa_\perp)} = I_{n_x}\otimes D_{yy}^{(\kappa_\perp)}, \quad \mathbf{H}_{y} = I_{n_x}\otimes H_y,
\end{align}
where $I_{n_x} \in \mathbb{R}^{n_x\times n_x}$, $I_{n_y}\in \mathbb{R}^{n_y\times n_y}$ are identity matrices and  $H_{x}, D_{x}, D_{xx}^{(\kappa_\perp)} \in \mathbb{R}^{n_x\times n_x}$, $H_{y}, D_{y}, D_{yy}^{(\kappa_\perp)}\in \mathbb{R}^{n_y\times n_y}$ are the 1D SBP operators and the diagonal norms from \S\ref{sub:sec:sbp_operators}.

To approximate integrals over the 2D computational plane we introduce the quadrature rules for integrable grid functions $\bm{u}$,
\begin{align}\label{eq:discrete_h_quad}
  \mathbf{H}=H_x\otimes H_y, \quad \bm{1}^T\mathbf{H}\bm{u} \approx \int_{\Omega} u \;\dd x \dd y,
\end{align}
and
\begin{align}\label{eq:discrete_l2_quad}
  \mathbf{I}_h=\Delta{x}\Delta{y}\left(I_{n_x}\otimes I_{n_y}\right), \quad  \bm{1}^T\mathbf{I}_h\bm{u} \approx \int_{\Omega} u \; \dd x \dd y, 
\end{align}
where $ \bm{1} = [1, 1, \cdots 1]^T.$
The quadrature rule \eqref{eq:discrete_h_quad} is given by the SBP operator used, and therefore inherits the accuracy of the underlying SBP operator. 
The latter \eqref{eq:discrete_l2_quad} corresponds to the standard 2D Riemann sum which is first order accurate for continuous functions.
We define the corresponding discrete norms,
\begin{align}\label{eq:discrete_l2_norm}
\|\bm{u}\|_H^2 = \bm{u}^T\mathbf{H}\bm{u} > 0 .
\quad
\|\bm{u}\|_{l_2}^2 = \bm{u}^T\mathbf{I}_h\bm{u}>0 . 
\end{align}
For a linear operator $\mathbf{P}: \mathbb{R}^{n_xn_y} \to \mathbb{R}^{n_xn_y}$, we also introduce the operator norms,
\begin{align}\label{eq:operator_norm}
\|\mathbf{P}\|_{l_2} = \max_{\bm{u} \in \mathbb{R}^{n_xn_y}}\frac{\|\mathbf{P}\bm{u}\|_{l_2}}{\|\bm{u}\|_{l_2}}, \quad
\|\mathbf{P}\|_H = \max_{\bm{u} \in \mathbb{R}^{n_xn_y}}\frac{\|\mathbf{P}\bm{u}\|_H}{\|\bm{u}\|_H}.
\end{align}
The discrete norms, $\|\cdot\|_{H}$ and $\|\cdot\|_{l_2}$,  are equivalent, that is there exist positive constants $\alpha, \beta > 0$, such that
\begin{align}\label{eq:equivalence_l2_norm}
\alpha\|\cdot\|_{H}\le \|\cdot\|_{l_2} \le \beta\|\cdot\|_{H}.
\end{align}

 To define the SAT penalty terms for boundary conditions we consider the  unit vectors $\bm{e}_1=(1, 0, \cdots, 0)^{T}, \bm{e}_n=(0, 0, \cdots, 1)^{T} \in \mathbb{R}^n$, and the boundary projection operators
\begin{align}\label{eq:1D boundary operators}
\begin{split}
    B_{1x} &= \bm{e}_{1}\bm{e}_{1}^{T}, \quad B_{n_x} = \bm{e}_{n_x}\bm{e}_{n_x}^{T}, \quad 
 \bm{e}_1, \bm{e}_{n_x} \in \mathbb{R}^{n_x},\\
 \quad B_{1y} &= \bm{e}_{1}\bm{e}_{1}^{T}, \quad B_{n_x} = \bm{e}_{n_y}\bm{e}_{n_x}^{T}, \quad 
 \bm{e}_1, \bm{e}_{n_y} \in \mathbb{R}^{n_y},
 \\
 E_{1}     &= \bm{e}_1\bm{e}_1^{T} - \bm{e}_1\bm{e}_{n_y}^{T},  \quad
        E_{n_y}   = \bm{e}_{n_y}\bm{e}_{n_y}^T - \bm{e}_{n_y}\bm{e}_{1}^{T}, \quad \bm{e}_1, \bm{e}_{n_y} \in \mathbb{R}^{n_y}.
\end{split}
\end{align}
Note that
\begin{align*}
B_{1x}\bm{v} &= \left(v_1, 0, \cdots, 0\right)^T,  \quad B_{n_x}\bm{v} = \left(0, \cdots, 0, v_{n_x}\right)^T,  \quad \forall \bm{v} \in \mathbb{R}^{n_x},
\\
E_{1}\bm{v} &= \left(v_1-v_{n_y}, 0, \cdots, 0\right)^T,  \quad   E_{n_y}\bm{v} = \left(0, \cdots, 0, v_{n_y}-v_1\right)^T, \quad \forall \bm{v} \in \mathbb{R}^{n_y}.
\end{align*}
As above, the boundary operators \eqref{eq:1D boundary operators} are extended to 2D using Kronecker products
\begin{align}\label{eq:boundary_operators}
\begin{split}
    \mathbf{B}_{1x} &= {B}_{1x}\otimes I_{n_y}, \quad \mathbf{B}_{n_x} = {B}_{n_x}\otimes I_{n_y}, \quad \mathbf{B}_{1y} = {I}_{n_x}\otimes B_{1y}, \quad \mathbf{B}_{n_y} = {I}_{n_x}\otimes B_{n_y}, \\
    \mathbf{E}_{1y} &= I_{n_x} \otimes E_{1} , \quad \mathbf{E}_{n_y} = I_{n_x} \otimes E_{n_y}.
\end{split}
\end{align}

\subsection{Semi-discrete perpendicular diffusion operator}\label{sub:sec:Discretisation}
A semi-discrete approximation of the perpendicular diffusion operator on the computational plane is obtained by replacing the continuous derivatives with the SBP operators giving,
\begin{align}\label{eq:perpendicular_diffusion}
    \nabla\cdot(\kappa_\perp\nabla u) \approx 
    \left(\mathbf{D}_{xx}^{(\kappa_\perp)} + \mathbf{D}_{yy}^{(\kappa_\perp)}\right)\bm{u}.
\end{align}
For the Dirichlet boundary conditions \eqref{eq:Boundary x Dirichlet} in the $x$ direction and periodic boundary conditions \eqref{eq:Boundary y periodic} in the  $y$ direction, the SATs are given by,
\begin{align}\label{eq:SAT dirichlet_x}
    \mathbf{SAT}_x &= \tau_{x,0} \kappa_\perp \mathbf{H}_x^{-1} \left(\mathbf{B}_{n_x}-\mathbf{B}_{1x}\right) \mathbf{H}_x^{-1} \left(\mathbf{B}_{n_x}-\mathbf{B}_{1x}\right) (\bm{u} - \bm{g}) \nonumber\\&\qquad+ \tau_{x,1} \kappa_\perp \mathbf{H}_x^{-1} (\mathbf{H}_x\mathbf{D}_x)^T \mathbf{H}_x^{-1} \left(\mathbf{B}_{n_x}-\mathbf{B}_{1x}\right) (\bm{u} - \bm{g}) \\
    \label{eq:SAT periodic_y}
    \mathbf{SAT}_y &= \tau_{y,0}\kappa_\perp\mathbf{H}_y^{-1} \left(\mathbf{E}_{1y} + \mathbf{E}_{n_y}\right) \bm{u} + \tau_{y,1}\kappa_\perp \mathbf{H}_y^{-1} \mathbf{D}_y^T \left(\mathbf{E}_{1y} - \mathbf{E}_{n_y}\right) \bm{u} \nonumber\\&\qquad+  \tau_{y,2}\kappa_\perp\mathbf{H}_y^{-1} \left(\mathbf{E}_{1y} - \mathbf{E}_{n_y}\right) \mathbf{D}_y \bm{u},
\end{align}
where the boundary projection operators $\mathbf{B}_{n_x},\mathbf{B}_{1x}$ and $\mathbf{E}_{n_x},\mathbf{E}_{1x}$ are given by \eqref{eq:boundary_operators}. Here $\bm{g}$ is a vector of data for the Dirichlet boundary conditions \eqref{eq:Boundary x Dirichlet} along the $x$ boundaries.
The real valued penalty parameters $\tau_{x,0}$, $\tau_{x,1}$, $\tau_{y,0}$, $\tau_{y,1}$, and $\tau_{y,2}$ are to be determined by stability requirements. 
Adding the SATs to the perpendicular operator \eqref{eq:perpendicular_diffusion} defines the discrete perpendicular diffusion operator with boundary conditions,
\begin{align}\label{eq:Perpendicular Operator}
     \mathbf{P}_\perp := 
    \left(\mathbf{D}_{xx}^{(\kappa_\perp)} + \mathbf{D}_{yy}^{(\kappa_\perp)}\right)\bm{u} + \operatorname{\textbf{SAT}}_x + \operatorname{\textbf{SAT}}_y.
\end{align}
The following theorem constrains the penalty parameters and states the stability properties of the discrete perpendicular operator.
\begin{theorem}\label{thm:spsd_perpendicular_diffution_operator}
    Consider the semi-discrete approximation of perpendicular diffusion operator 
     $\mathbf{P}_\perp$ defined in  \eqref{eq:Perpendicular Operator} with the SATs given by \eqref{eq:SAT dirichlet_x} and \eqref{eq:SAT periodic_y}. Let $\mathbf{H}=(H_x\otimes H_y)$, where $H_x=\Delta{x}\emph{diag}([h_1,h_2,\cdots,h_{n_x}])$ and $H_y=\Delta{y}\emph{diag}([h_1,h_2,\cdots,h_{n_y}])$ with $h_j >0$. If $\bm{g} =0$ and  $\tau_{x,1}=-1$, $\tau_{x,0}=-(1+\tau_{x,2})$, $\tau_{x,2}\ge 0$, $\tau_{y,0}=-\kappa_{\perp}/2h_{1}$ and $\tau_{y,1}=-\tau_{y,2}=\half$,
    then
    \begin{align}\label{eq:spsd_perpendicular_diffution_operator}
        \mathbf{P}_\perp = - \mathbf{H}^{-1} \mathbf{A}_\perp, \quad \mathbf{A}_\perp = \mathbf{A}_\perp^T, \quad \bm{u}^T\mathbf{A}_\perp\bm{u} \ge 0, \quad \forall \bm{u}\in\mathbb{R}^{n_x n_y}.
    \end{align}
\end{theorem}
The proof can be found in \ref{sec:appendix:extra mats}.
\begin{remark}
    Other stable methods for implementing boundary conditions, such as by injection for Dirichlet boundaries, may also be used. We have chosen to implement SATs so that we can generalise the method to multi-block schemes and complex geometries. We refer interested readers to references \cite{olsson_summation_1995,olsson_summation_1995-1,svard_review_2014} for more detail on stable methods for boundary condition enforcement.
\end{remark}
Next we turn our attention to the numerical approximation of the parallel diffusion operator.

\subsection{The parallel diffusion operator}\label{sec:sub:Parallel operator}
We will discuss the parallel diffusion operator modelled by field line tracing and how to approximate it numerically in a stable manner.
\subsubsection{Continuous parallel diffusion operator}
The continuous parallel operator $\mathcal{P}_\parallel$ diffuses the solution along the magnetic field. To model this we use field line tracing, which involves following the trajectory of individual field lines by solving the ordinary differential equation,
\begin{align}\label{eq:Field line ODE}
    \dv{\bm{x}}{\phi} = \bm{B}(\bm{x}), \qquad \bm{x}(0) = \bm{x}_0 = (x_0,y_0,\phi_0),
\end{align}
where $\phi$ is treated as a time-like variable for the purposes of solving the ODE and $\bm{B}:\R^3\to\R^3$ \cite{meiss_canonical_1990}.
Since the domain is periodic and we are considering the solution on a single $z=const$ plane, this corresponds to setting $\phi_0=z$ when $z\in[0,2\pi)$.

\begin{remark}
    One may also re-parameterise the equations in terms of the arc-length so that $\dv{\bm{x}}{s} = \bm{b}(\bm{x}(s))$ where $\bm{b}=\bm{B}/\|\bm{B}\|$ and $\bm{x}(s=0) = \bm{x}_0$.
\end{remark}

Consider a single field line defined by the ODE \eqref{eq:Field line ODE} with initial condition $\bm{x}_0\in\Omega$, and its position after integration to $\bm{x}_+ = (x_+,y_+,\phi_+)$ where $\phi_+=\phi_0+2\pi$ corresponds to returning to the $z=const$ plane. Now let $\mathcal{P}_f$ be an operator which models the diffusion `forward' along field lines by mapping the solution $\mathcal{P}_f u(\bm{x}_0,t)=u(\bm{x}_+,t)$. The `backwards' diffusion can similarly be modelled the diffusion backwards along field lines by integrating from $\phi_0\to\phi_0-2\pi$, so that $\mathcal{P}_b u(\bm{x}_0,t)=u(\bm{x}_-,t)$. 
Then the full parallel operator can be written $\mathcal{P}_\parallel=\half(\mathcal{P}_f+\mathcal{P}_b)$.

The Poincar\'e section shown in Figure \ref{fig:field line tracing} is constructed by solving the ODE \eqref{eq:Field line ODE} and placing a `dot' each time the field line returns to the given $\phi=\phi_0$ plane.

\begin{figure}[H]
    \centering
        \includegraphics[width=\columnwidth]{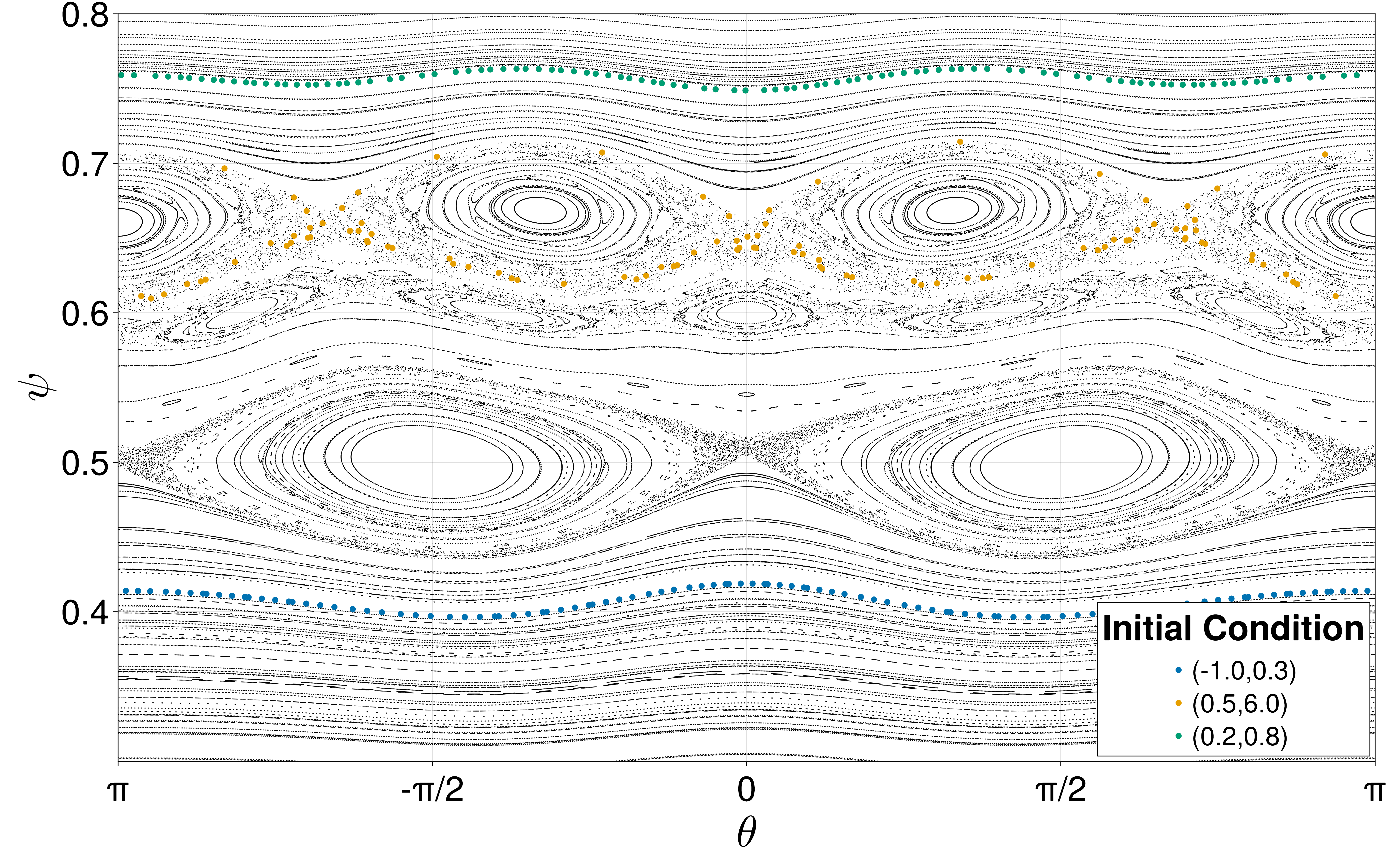}
    \caption{
        Example Poincar\'e section of a magnetic field in a slab. The feature of the field such as the islands can be controlled by applying a perturbation to the field line Hamiltonian that yields the system of ODEs \eqref{eq:Field line ODE}. The trajectories of three different field lines are shown. 
        }
    \label{fig:field line tracing}
\end{figure}

To understand the action of the parallel diffusion consider the sample field in Figure \ref{fig:field line tracing}. 
The solution is expected to flatten along field line transport barriers, such as those highlighted by the green and blue points. Any holes in the barriers which may let field lines cross will still allow the fast scale transport to `leak' through though, so the solution should be relatively constant along these barriers.
In the presence of islands and chaos, i.e. the region filled with yellow points, the solution is expected to flatten across the entire region due to excursions in the $\psi$ direction.
In the time dependant case the flattening in chaotic regions may take a long time, since it may require many transits for the field lines to distribute the solution.

In this paper we assume that field lines are confined (physically equivalent to assuming field lines do not strike a wall for instance), however the method can be adapted to handle non-confined field lines, by providing a suitable parallel mapping function if desired. A future work will examine the case where field lines are terminated ``early'' (without returning to the Poincar\`e section) by a wall or boundary.
Except for the ``NIMROD'' test \cite{sovinec_nonlinear_2004}, we also assume points leave the computational plane and return to it. In the ``NIMROD'' case, the problem is purely synthetic and designed to test the properties of the numerical parallel operator.

\begin{remark}
    In \cite{muir_numerical_2021} we showed that in the 2D case, that any set of points generated by a contractive map will be stable, and can be used to construct a parallel map.
\end{remark}

\subsubsection{Numerical parallel diffusion treatment}

In a hollow torus (in the present case this is a periodic box), the process of constructing the parallel operator is analogous to the discussion in the section above. Here, the grid points \eqref{eq:CartesianGrid} are used as initial conditions to the ODE \eqref{eq:Field line ODE} and traced backwards and forwards to the same plane. The forward and backward `grid' formed by the field line tracing will likely not correspond to a regular grid and so the values of $\bm{w}_f$ and $\bm{w}_b$ are determined by interpolation using the data on the original grid \eqref{eq:CartesianGrid}.

In the following discussion we consider the parallel penalty in terms of tracing field lines between $z$ cross sections of a 3D domain. 

The numerical tests presented in \S\ref{sec:Numerical Results} all use the \textit{Tsit5} integrator provided by the \textit{DifferentialEquations.jl}\footnote{\href{https://docs.sciml.ai/DiffEqDocs/stable/}{docs.sciml.ai/DiffEqDocs/stable/}} Julia package~\cite{rackauckas_differentialequationsjl_2017}. In practice we have found that setting a relative tolerance of $10^{-6}$ or less for the adaptive stepping ensures good convergence results.

\begin{remark}
    As mentioned, the forward and backward maps can in principle be generated by any function, as long as it is contractive.
\end{remark}

\begin{remark}
    One advantage of this method is that the magnetic field function can be defined in any way. One may also define the magnetic field in terms of arclength, ``toroidal length''.
\end{remark}

The action of the parallel diffusion operator on a Cartesian grid can be interpreted numerically as projection operators defined by the matrix products 
\begin{align}\label{eq:parallel_projection}
    \bm{P}_f = \bm{\Pi}_f \cdot \mathbb{P}_f, \quad  \bm{P}_b = \bm{\Pi}_b \cdot\mathbb{P}_b.
\end{align}
Here $\bm{\Pi}_f$, $\bm{\Pi}_b$ are stable interpolation matrices and $\mathbb{P}_f$,  $\mathbb{P}_b$ are permutation/distribution matrices. The permutation matrices are implicitly provided by following a point in the plane along the field line and recording its position as it lands back on the plane. If the point does not fall on a grid point, then the solution is computed by interpolation from neighbouring grid points. The particular interpolation scheme can be chosen arbitrarily.
In this work we will consider interpolation by stable cubic splines such that $\|\bm{\Pi}_f\|_{l_2}\leq1$, $\|\bm{\Pi}_b\|_{l_2}\leq1$, \cite{hong-ci_stability_1983}.

The following Lemma will be used to constrain the parallel numerical diffusion operator and will be useful in proving the stability of the numerical method.
\begin{lemma}\label{lem:bounding projection operator}
    Consider the forward $\bm{P}_f$ and backward $\bm{P}_b$ parallel diffusion projection operators given by \eqref{eq:parallel_projection} with $\bm{P}_f = \bm{\Pi}_f \cdot \mathbb{P}_f$, $\bm{P}_b = \bm{\Pi}_b \cdot \mathbb{P}_b$, where $\bm{\Pi}_f$, $\bm{\Pi}_b$ are stable interpolation matrices and $\mathbb{P}_f$,  $\mathbb{P}_b$ are permutation matrices. For all stable interpolation matrices with $\|\bm{\Pi}_f\|_{l_2}\leq1$ and $\|\bm{\Pi}_b\|_{l_2}\leq1$ then we have $\|\bm{P}_f\|_{l_2}\leq1$ and $\|\bm{P}_b\|_{l_2}\leq1$.
\end{lemma}
\begin{proof}
Recall that for permutation matrices $\mathbb{P}_f$, $\mathbb{P}_b$ we have $\|\mathbb{P}_f\|_{l_2} = 1$, $\|\mathbb{P}_b\|_{l_2} = 1$.
    Therefore the parallel diffusion projection operators yield
    $$
\|\bm{P}_f\|_{l_2} = \|\bm{\Pi}_f\cdot \mathbb{P}_f\|_{l_2} \le \|\bm{\Pi}_f\|_{l_2}  \|\mathbb{P}_f\|_{l_2} = \|\bm{\Pi}_f\|_{l_2} \le 1,
    $$
    and
    $$
\|\bm{P}_b\|_{l_2} = \|\bm{\Pi}_b \cdot \mathbb{P}_b\|_{l_2} \le \|\bm{\Pi}_b\|_{l_2}  \|\mathbb{P}_b\|_{l_2} = \|\bm{\Pi}_b\|_{l_2} \le 1.
    $$
\end{proof}

The numerical procedure for imposing the parallel solution on the grid is performed weakly via a parallel penalty term on the entire computational plane. That is
$$
\bm{u} - \frac{1}{2}\left(\bm{P}_b\bm{u} + \bm{P}_f\bm{u}\right)=0 \iff \mathbf{P}_\parallel \bm{u} =0,
$$
where the numerical parallel diffusion operator is given by,
\begin{align}\label{eq:discrete parallel diffusion operator}
    \mathbf{P}_\parallel =-  \tau_\parallel \kappa_\parallel \bm{H}^{-1} \left(\bm{I} - \frac{1}{2}[\bm{P}_f + \bm{P}_b]\right).
\end{align}
Here $\bm{I}$ is the identity matrix and $\tau_\parallel>0$ is a scalar chosen to ensure stability. 
Note in particular, if the magnetic field lines are aligned with the grid, then $\bm{P}_f = \bm{P}_b = I$, and the parallel operator vanishes completely, i.e. $\mathbf{P}_\parallel=\bm{0}$.

We will now prove the theorem which shows that the numerical parallel  diffusion operator given in \eqref{eq:discrete parallel diffusion operator} will not contribute to energy growth.

\begin{theorem}\label{theo:parallel_operator}
    Consider the discrete form of the parallel diffusion operator given by \eqref{eq:discrete parallel diffusion operator}. If $\|\bm{P}_f\|_{l_2}\leq1$, $\|\bm{P}_b\|_{l_2}\leq1$, $\kappa_\parallel>0$ and $\tau_\parallel>0$, then,
    $$\bm{u}^T(\bm{H}\bm{P}_\parallel + (\bm{H}\bm{P}_\parallel)^T)\bm{u} \leq 0, \quad \forall \bm{u}\in\R^n.$$
\end{theorem}
\begin{proof}
    Consider $\bm{u}^T\bm{H}\bm{P}_\parallel\bm{u}$ and add the transpose of the product,
    \begin{align*}
        \bm{u}^T(\bm{H}\bm{P}_\parallel + (\bm{H}\bm{P}_\parallel)^T)\bm{u} 
            &= -\tau_\parallel \kappa_\parallel \bm{u}^T \left(2\bm{I} - \half[\bm{P}_f + \bm{P}_f^T] - \half[\bm{P}_b + \bm{P}_b^T] \right) \bm{u}.
    \end{align*}
    Let $\bm{A}_f=\half\left(\bm{P}_f + \bm{P}_f^T\right)$ and $\bm{A}_b=\half\left(\bm{P}_b + \bm{P}_b^T\right)$ such that,
    \begin{align*}
        \bm{u}^T\left(\bm{H}\bm{P}_\parallel + \left(\bm{H}\bm{P}_\parallel\right)^T\right)\bm{u} 
        &= -\tau_\parallel \kappa_\parallel \bm{u}^T\left(2\bm{I} - \bm{A}_f - \bm{A}_b\right)\bm{u}.
    \end{align*}
    Since $\|\bm{P}_b\|,\|\bm{P}_f\|\leq 1$ and $\bm{A}_f$ and $\bm{A}_b$ are symmetric indefinite matrices, it follows that,
    \begin{align*}
        -4\tau_\parallel\kappa_\parallel \|\bm{u}\|^2 \leq
        -\tau_\parallel \kappa_\parallel \bm{u}^T\left(2\bm{I} - \bm{A}_f - \bm{A}_b\right)\bm{u} \leq
        0.
    \end{align*}
    Hence, 
    \begin{align*}
        \bm{u}^T(\bm{H}\bm{P}_\parallel + (\bm{H}\bm{P}_\parallel)^T)\bm{u} \leq 0,
    \end{align*}
    which completes the proof.
\end{proof}

We will choose the nonlinear penalty
\begin{align}\label{eq:accurate_parallel_penalty}
  \tau_\parallel = \alpha\left(\frac{\|\bm{u}-\bm{w}\|_\infty}{\|\bm{u}\|_\infty}\right)^\beta > 0,  \quad \alpha>0,\, \beta\geq0,
\end{align}
which ensures stability and effectively minimises round off errors when $\bm{u}\sim \bm{w}$. Moreover, the inclusion of relative difference ensures that the parallel diffusion operator is enforced more strongly when the field is far from equilibrium.
Inclusion of the $\bm{H}^{-1}$ operator also ensures that the error generated by the parallel diffusion operator \eqref{eq:discrete parallel diffusion operator} scales as $O(\Delta x\Delta y)$ for all stable interpolation operators.
In practice we have found that $\alpha=0.1$ and $\beta=2$ provides good convergence results for values of $\kappa_\parallel$ up to $10^{10}$. 

It is significantly noteworthy that the accuracy of the numerical parallel diffusion operator is closely tied to the parallel penalty $\tau_\parallel$, the interpolation operators and the accuracy of  ODE solver for \eqref{eq:Field line ODE}. More interestingly, the accuracy of the numerical parallel diffusion operator is  independent of the SBP finite difference operators used for approximating the perpendicular diffusion operator.

\begin{remark}
  The construction of the parallel operator requires solving the  $2n_xn_y$ system of ODEs \eqref{eq:Field line ODE}. This can be expensive for time dependant fields, but can be computed once at the beginning of the simulation and saved when a time-independent magnetic field is given.
\end{remark}

\subsection{Semi-discrete approximation}\label{sec:sub:Semi-discrete approximation}
The semi-discrete approximation of the field aligned anisotropic diffusion equation \eqref{eq:ADE field aligned}, with initial condition \eqref{eq:ADE initial condition} and the boundary conditions  \eqref{eq:Boundary x Dirichlet} and \eqref{eq:Boundary y periodic} is
\begin{align}\label{eq:semi-discrete_ADE}
\frac{d\bm{u}}{dt} = \mathbf{P}_\perp \bm{u}  + \mathbf{P}_\parallel \bm{u} + \bm{F}(t), \quad \bm{u}(0) = \bm{f}.
\end{align}
Here $\mathbf{P}_\perp$ and $\mathbf{P}_\parallel$ are the perpendicular and parallel numerical diffusion operators given by \eqref{eq:Perpendicular Operator} and  \eqref{eq:discrete parallel diffusion operator}, respectively.
The following theorem proves that the semi-discrete approximation \eqref{eq:semi-discrete_ADE} is stable.
\begin{theorem}\label{theo:semi-discrete_stability}
    Consider the semi-discrete field aligned anisotropic diffusion equation \eqref{eq:semi-discrete_ADE}, where $\mathbf{P}_\perp$ and $\mathbf{P}_\parallel$ are the perpendicular and parallel numerical diffusion operators given by \eqref{eq:Perpendicular Operator} and  \eqref{eq:discrete parallel diffusion operator} with initial condition \eqref{eq:ADE initial condition} and with boundary conditions given by \eqref{eq:Boundary x Dirichlet} and \eqref{eq:Boundary y periodic}. For $\bm{g} =0$ and  $\bm{F} =0$, if  
    \begin{align*}
        \bm{u}^T\left((\mathbf{H}\mathbf{P}_\perp) + (\mathbf{H}\mathbf{P}_\perp)^T\right)\bm{u}\leq0,\quad \bm{u}^T\left((\mathbf{H}\mathbf{P}_\parallel) + (\mathbf{H}\mathbf{P}_\parallel)^T\right)\bm{u}\leq0,\quad \forall\bm{u}\in\R^{n_xn_y}.
    \end{align*}
    then we have
    \begin{align}
    {\|\bm{u}(t)\|_H \le \|\bm{f}\|_H}, \quad \forall t \ge 0,  \quad \|\bm{u}\|_H^2 = \bm{u}^T\mathbf{H}\bm{u} > 0 .
    \end{align}
\end{theorem}
\begin{proof}
    We set $\bm{F} =0$ and from the left multiply \eqref{eq:semi-discrete_ADE} with $\bm{u}^T\mathbf{H}$, adding the transpose of the product giving
    \begin{align*}
        \dv{t}\|\bm{u}\|_H^2  = \bm{u}^T\left((\mathbf{H}\mathbf{P}_\perp) + (\mathbf{H}\mathbf{P}_\perp)^T\right)\bm{u} +  \bm{u}^T\left((\mathbf{H}\mathbf{P}_\parallel) + (\mathbf{H}\mathbf{P}_\parallel)^T\right)\bm{u}\leq0.
    \end{align*}
    Which holds by combining theorems \eqref{thm:spsd_perpendicular_diffution_operator} and \eqref{theo:parallel_operator}. The proof is complete.
\end{proof}
Theorem \ref{theo:semi-discrete_stability} is completely analogous to the continuous counterpart Theorem \ref{theo:well-posedness}. We conclude this analysis by stating the stability of the semi-discrete approximation \eqref{eq:semi-discrete_ADE} is established by Theorems \ref{thm:spsd_perpendicular_diffution_operator}, \ref{theo:parallel_operator} and \ref{theo:semi-discrete_stability}.

\subsection{Fully-discrete approximation}\label{sec:sub:fully-discrete approximation}
We discretise the time variable $t \in [0, T]$ with the time-step $\Delta t_l>0$ so that $t_{l+1}=t_l + \Delta t_l$ where $t_0 =0$ and $l=0,1,2,\cdots$ and $T>0$ is the final time. The numerical solution at $t = t_l$ is denoted $\bm{u}^l$.

We will approximate \eqref{eq:semi-discrete_ADE} in time using a splitting technique.
Introducing an intermediate state $\bm{u}^{l+\half}$ which solves the perpendicular diffusion problem, we have
\begin{align}\label{eq:Fully Discrete stage 1}
    \left(\bm{I} - \theta\Delta t\bm{P}_\perp\right)\bm{u}^{l+\half} = \left( \bm{I} + (1-\theta)\Delta t\bm{P}_\perp \right)\bm{u}^l + \theta\Delta t\bm{F}^{l} + (1-\theta)\Delta t\bm{F}^{l+\half},
\end{align}
\begin{subequations} \label{eq:Fully Discrete stage 2}
\begin{align}\label{eq:Fully Discrete stage 2.1}
    \bm{w}_b^{l+\half} &= \bm{P}_b\bm{u}^{l+\half}, \quad  \bm{w}_f^{l+\half} = \bm{P}_f\bm{u}^{l+\half},\\
   \bm{u}^{l+1} &= \bm{u}^{l+\half} -
        \Delta t\tau_\parallel\kappa_\parallel \bm{H}^{-1} \left(\bm{u}^{l+1} - \half[\bm{w}_f^{l+\half} + \bm{w}_b^{l+\half}]\right).
   \label{eq:Fully Discrete stage 2.2}
\end{align}
\end{subequations}
For $\theta=0.5$ this amounts to using Crank-Nicholson for the perpendicular diffusion \eqref{eq:Fully Discrete stage 1} and implicit midpoint rule for the parallel diffusion \eqref{eq:Fully Discrete stage 2.2}.
The first stage \eqref{eq:Fully Discrete stage 1} propagates the perpendicular diffusion through the grid.  The corresponding linear algebraic problem \eqref{eq:Fully Discrete stage 1} can be solved efficiently using the conjugate gradient method, see \ref{sec:conjugate_gradient} for details.
The second stage \eqref{eq:Fully Discrete stage 2.1}-\eqref{eq:Fully Discrete stage 2.2} uses an implicit midpoint rule to propagate the parallel diffusion. This requires field line tracing \eqref{eq:Fully Discrete stage 2.1}, and uses the Cartesian grid points as initial conditions to the ODE \eqref{eq:Field line ODE}. The parallel diffusion solution is propagated through the grid by the parallel penalty term given by \eqref{eq:Fully Discrete stage 2.2}. The parallel solve \eqref{eq:Fully Discrete stage 2.2} involves a trivial linear algebraic problem.
In particular, we note that
\begin{align}\label{eq:Fully Discrete stage 2.2.1}
    \begin{split}
        \bm{u}^{l+1} = (\bm{I} + \Delta t\tau_\parallel\kappa_\parallel \bm{H}^{-1})^{-1} \left(\bm{u}^{l+\half} + \frac{\Delta t\tau_\parallel\kappa_\parallel}{2}\bm{H}^{-1}\bm{w}^{l+\half}\right).
    \end{split}
\end{align}

We will now prove that our splitting technique \eqref{eq:Fully Discrete stage 1}-\eqref{eq:Fully Discrete stage 2} is unconditionally stable.
\begin{theorem}\label{theo:fully-discrete_stability_split_form}
    Consider the fully-discrete field aligned anisotropic diffusion equation in split form \eqref{eq:Fully Discrete stage 1}-\eqref{eq:Fully Discrete stage 2}  with the initial condition $\bm{u}^{0} = \bm{f}$, where $\mathbf{P}_\perp$ and $\mathbf{P}_\parallel$ are the perpendicular and parallel numerical diffusion operators given by \eqref{eq:Perpendicular Operator} and \eqref{eq:discrete parallel diffusion operator}. For $\bm{g}=0$, $\bm{F}=0$ and $\half\le \theta\le 1$, if  
    \begin{align*}
        \bm{u}^T\left((\mathbf{H}\mathbf{P}_\perp) + (\mathbf{H}\mathbf{P}_\perp)^T\right)\bm{u}\leq0,\quad \|\bm{P}_f\|\leq1, \quad  \|\bm{P}_b\|\leq1,
    \end{align*}
    then we have
    \begin{align}
        \|\bm{u}^{l}\| \le \|\bm{f}\|, \quad \forall \, l \ge 0, \quad \forall \, \Delta t_l>0 .
    \end{align}
\end{theorem}
\begin{proof}
    Multiplying equation \eqref{eq:Fully Discrete stage 1} with $\left(\bm{u}^{l+\half}\right)^T\bm{H}$ from the left and adding the transpose of the product gives,
    \begin{align*}
        &\left(\bm{u}^{l+\half}\right)^T \left(2\bm{H} - \theta\Delta t [(\bm{H}\bm{P}_\perp)^T + (\bm{H}\bm{P}_\perp)]\right)\bm{u}^{l+\half} = \\
        &\qquad\qquad\left( \bm{u}^{l+\half} \right)^T (2\bm{H} + (1-\theta)\Delta t[(\bm{H}\bm{P}_\perp)^T + (\bm{H}\bm{P}_\perp)])\bm{u}^l.
    \end{align*}
    Letting $\bm{G}_\perp=\bm{H} -\theta/2\left(\bm{H}\bm{P}_\perp +  (\bm{H}\bm{P}_\perp)^T\right) = \bm{H} + \theta\bm{A}_\perp$ with $\half\le \theta\le 1$ then,
    \begin{align*}
    2\|\bm{u}^{l+\half}\|_{\bm{G}_\perp}^2 &\leq 
        2\|\bm{u}^{l+\half}\|_{\bm{G}_\perp}\|\bm{u}^{l}\|_{\bm{G}_\perp}, 
    \end{align*}
    giving 
    $$
    \|\bm{u}^{l+\half}\|_{\bm{G}_\perp} \leq \|\bm{u}^l\|_{\bm{G}_\perp}
    \iff
    \|\bm{u}^{l+\half}\|_{\bm{H}} \leq \|\bm{u}^l\|_{\bm{H}}.
    $$
    
    Next consider equation \eqref{eq:Fully Discrete stage 2.2.1}, multiply by $(\bm{u}^{l+1})^T(\bm{H} + \Delta t \tau_\parallel\kappa_\parallel\bm{I})$ and add the transpose of the product, we have
    \begin{align*}
        2(\bm{u}^{l+1})^T(\bm{H} + \Delta t \tau_\parallel\kappa_\parallel\bm{I})\bm{u}^{l+1} = (\bm{u}^{l+1})^T\left(2\bm{H} + \Delta t\tau_\parallel\kappa_\parallel\half\left([\bm{P}_f + \bm{P}_b] + [\bm{P}_f + \bm{P}_b]^T\right)\right)\bm{u}^{l+\half}.
    \end{align*}
    Note that $\|\bm{I}\| \geq \|\half(\bm{P}_f + \bm{P}_b)\|$ follows from Lemma \ref{lem:bounding projection operator}. Hence,
    \begin{align*}
        2(\bm{H} + \Delta t \tau_\parallel\kappa_\parallel\bm{I}) \geq \left(2\bm{H} + \Delta t\tau_\parallel\kappa_\parallel\half\left([\bm{P}_f + \bm{P}_b] + [\bm{P}_f + \bm{P}_b]^T\right)\right).
    \end{align*}
    Therefore replacing the right hand side and noting it is symmetric positive definite, then setting $\bm{G}_\parallel=(\bm{H} + \Delta t \tau_\parallel\kappa_\parallel\bm{I})$ we have,
    \begin{align*}
      \|\bm{u}^{l+1}\|_{\bm{G}_\parallel}^2 \leq \|\bm{u}^{l+1}\|_{\bm{G}_\parallel}\|\bm{u}^{l+\half}\|_{\bm{G}_\parallel} &\implies \|\bm{u}^{l+1}\|_{\bm{G}_\parallel} \leq \|\bm{u}^{l+\half}\|_{\bm{G}_\parallel}\\
      &\iff \|\bm{u}^{l+1}\|_{\bm{H}} \leq \|\bm{u}^{l+\half}\|_{\bm{H}} \leq \|\bm{u}^l\|_{\bm{H}},
    \end{align*}
    and
    \begin{align*}
       \|\bm{u}^l\|_{\bm{H}} \le \|\bm{f}\|_{\bm{H}}, \quad \forall \, l\ge 0, \quad \forall \, \Delta t_l>0 .
    \end{align*}
    This completes the proof.
\end{proof}

Also note that, as stated in \S\ref{sec:sub:Parallel operator}, if $\bm{P}_\parallel=0$ $\bm{P}_b = \bm{P}_f = \bm{I}$, we have
\begin{align}
\bm{w}_b^{l+\half} = \bm{u}^{l+\half}, \quad  \bm{w}_f^{l+\half} = \bm{u}^{l+\half},
\end{align}
giving the expected solution $\bm{u}^{l+1}=\bm{u}^{l+\half}$, $\forall l\ge 0$, which is independent of the parallel diffusion.

\section{Numerical Results}\label{sec:Numerical Results}

In this section, we present some numerical results. Several numerical tests considered are designed to verify the accuracy and convergence properties of the different parts of our method, including the perpendicular and parallel numerical diffusion operators. We will use the method of manufactured solutions to test the perpendicular component exclusively with and without a magnetic field. This is followed by the ``NIMROD benchmark'' \cite{sovinec_nonlinear_2004} to show that the full scheme is asymptotic preserving. A self convergence test with a single island is also presented.
To end the section, we will present simulations involving magnetic field, with chaotic regions and islands, and show the contours of the numerical solutions of the anisotropic diffusion equation reproduce key features in the field.

\subsection{Perpendicular solve without magnetic field}\label{sec:sub:Code verification by method of manufactured solutions}
To verify the convergence properties of the perpendicular numerical diffusion operator we force the diffusion equation, with $\bm{P}_\parallel=0$, to have the exact solution,
\begin{align}\label{eq:Manufactured Solution}
    \widetilde{u}(x,y,t) = \cos(2\pi\omega_t t){\sin(2\pi\omega_x x+c_x)\sin(2\pi\omega_y y+c_y)},
\end{align}
where the values of the parameters $\omega_t$, $\omega_x$, $\omega_y$, $c_x$ and $c_y$  for the convergence tests are listed in Table \ref{tab:MMS parameters}.
The parameters for the spatial test have been chosen such that the errors due to spatial approximations dominate the temporal errors. Similarly, the parameters for the  temporal test are also chosen such that the errors due to temporal approximations dominate the spatial errors.
\begin{table}[H]
    \centering
    \caption{Parameters for convergence tests in Figure \ref{fig:MMS convergence}.}
    \begin{tabular}{cccccc}\hline
         Test       & $c_x$ & $c_y$ & $\omega_x$ & $\omega_y$ & $\omega_t$ \\\hline
         Spatial    & $1$   & $0$   & $7.5$ & $5$ & $1$ \\
         Temporal   & $1$   & $0$   & $1$   & $1$ & $9$ \\\hline
    \end{tabular}
    \label{tab:MMS parameters}
\end{table}

We consider the diffusion equation \eqref{eq:ADE field aligned} without the parallel operator, i.e. $\bm{P}_\parallel=0$. This is equivalent to having a magnetic field purely mesh aligned. The source term and the initial condition are given by,
\begin{align}
    F(x,y, t) = \pdv{\widetilde{u}}{t} - \nabla\cdot(\kappa_\perp\nabla\widetilde{u}), \quad u(x, y, 0) = \widetilde{u}(x,y,0).
\end{align}
We set the Dirichlet boundary condition \eqref{eq:Boundary x Dirichlet} in the $x$-direction, with the boundary data $g_L(y,t) = \widetilde{u}(x_L, y, t)$, $g_R(y,t) = \widetilde{u}(x_R, y, t)$, and the periodic boundary condition \eqref{eq:Boundary y periodic} in the $y$-direction.

Figure \ref{fig:MMS convergence} shows the convergence of $l_2$ errors using the implicit Euler ($\theta=1$) and Crank-Nicholson ($\theta=\half$) for the time solvers with $\Delta t = 0.1\Delta x$ in a $[0,1]\times[0,1]$ domain, and for SBP operators with 2nd and 4th order accurate interior stencils. The errors converge to zero at the expected rates, see Figure \ref{fig:MMS convergence}.

\begin{figure}[H]
    \begin{tikzpicture}
    \begin{groupplot}[group style={
            group size=2 by 3,
            x descriptions at=edge bottom,
            vertical sep    =5pt,
            horizontal sep  =5pt
        },
        my style,
        %
        enlargelimits=true,
        height=0.35\textwidth,
    ]
    \nextgroupplot[my legend style, 
        ymode=log,xmode=log,
        xtick={21,31,41,51,61,71,81,91,101},
        ylabel=log(relative error),
        grid]
        \addplot table[x index=0,y index=1,col sep=tab] {data_MMS_DP_spatial_Tests_theta10.csv};
        \addplot table[x index=0,y index=2,col sep=tab] {data_MMS_DP_spatial_Tests_theta10.csv};

        \node[] at (rel axis cs: 0.75,0.9) {\large $\theta=1$};

        \logLogSlopeTriangle{0.6}{0.4}{0.87}{2}{blue,dashed};
        \logLogSlopeTriangle{0.6}{0.4}{0.75}{3}{blue,dashed};
        
    \nextgroupplot[ymode=log,xmode=log,
        yticklabel pos=right,
        xtick={21,31,41,51,61,71,81,91,101},
        ylabel=Spatial error,
        grid]
        \addplot table[x index=0,y index=1,col sep=tab] {data_MMS_DP_spatial_Tests_theta05.csv};
        \addplot table[x index=0,y index=2,col sep=tab] {data_MMS_DP_spatial_Tests_theta05.csv};

        \node[] at (rel axis cs: 0.75,0.9) {\large $\theta=0.5$};

        \logLogSlopeTriangle{0.6}{0.4}{0.87}{2}{blue,dashed};
        \logLogSlopeTriangle{0.6}{0.4}{0.75}{3}{blue,dashed};

    \nextgroupplot[ymode=log,xmode=log,
            xtick={21,31,41,51,61,71,81,91,101},
            xticklabels={21,,41,,61,,,,101},
            xlabel=$n$,
            ylabel=log(relative error),
            grid]
            \addplot table[x index=0,y index=2,col sep=tab] {data_MMS_DP_temporal_Tests_theta10.csv};
            \addplot table[x index=0,y index=2,col sep=tab] {data_MMS_DP_temporal_Tests_theta10.csv};
            
            \node[] at (rel axis cs: 0.75,0.9) {\large $\theta=1$};

            \logLogSlopeTriangle{0.6}{0.4}{0.89}{1}{blue,dashed};
    \nextgroupplot[ymode=log,xmode=log,
            xtick={21,31,41,51,61,71,81,91,101},
            xticklabels={21,,41,,61,,,,101},
            yticklabel pos=right,
            ylabel=Temporal error,
            xlabel=$n$,
            grid]
            \addplot table[x index=0,y index=1,col sep=tab] {data_MMS_DP_temporal_Tests_theta05.csv};
            \addplot table[x index=0,y index=2,col sep=tab] {data_MMS_DP_temporal_Tests_theta05.csv};
            
            \node[] at (rel axis cs: 0.75,0.9) {\large $\theta=0.5$};

            \logLogSlopeTriangle{0.6}{0.4}{0.82}{2}{blue,dashed};

    \end{groupplot}

    \end{tikzpicture}
    \caption{
    Convergence rates of 2D code using method of manufactured solutions for the second order (blue) and fourth order (red) SBP operators. In each case $\Delta t= 0.1 \Delta x$, with $\Delta x=\Delta y$.
    Top: Spatial convergence. Bottom: Temporal convergence.
    Labels in the top right of each figure indicated the value of $\theta$ used.
    }
    \label{fig:MMS convergence}
\end{figure}
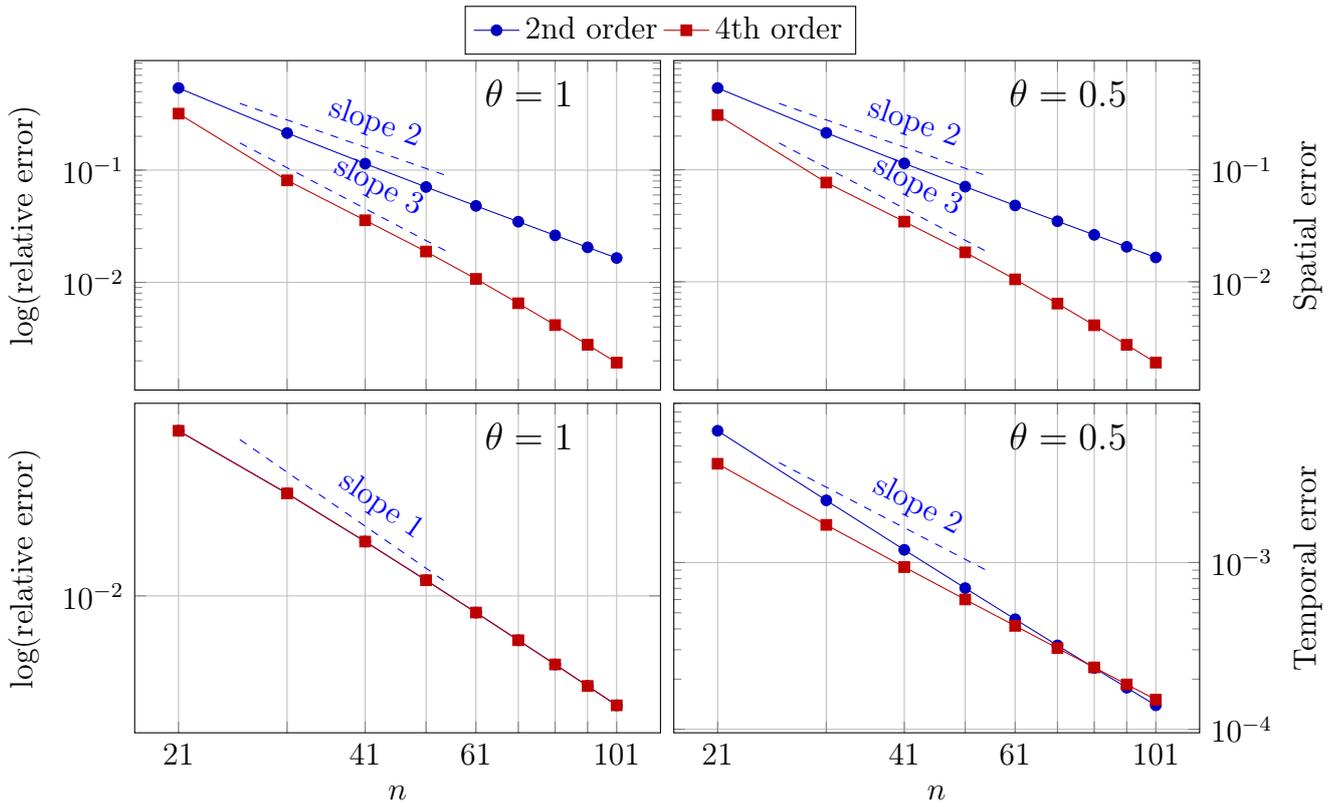

\subsection{NIMROD benchmark}\label{sec:sub:NIMROD benchmark}

\begin{figure}[H]
    \centering
    \includegraphics[width=0.8\textwidth]{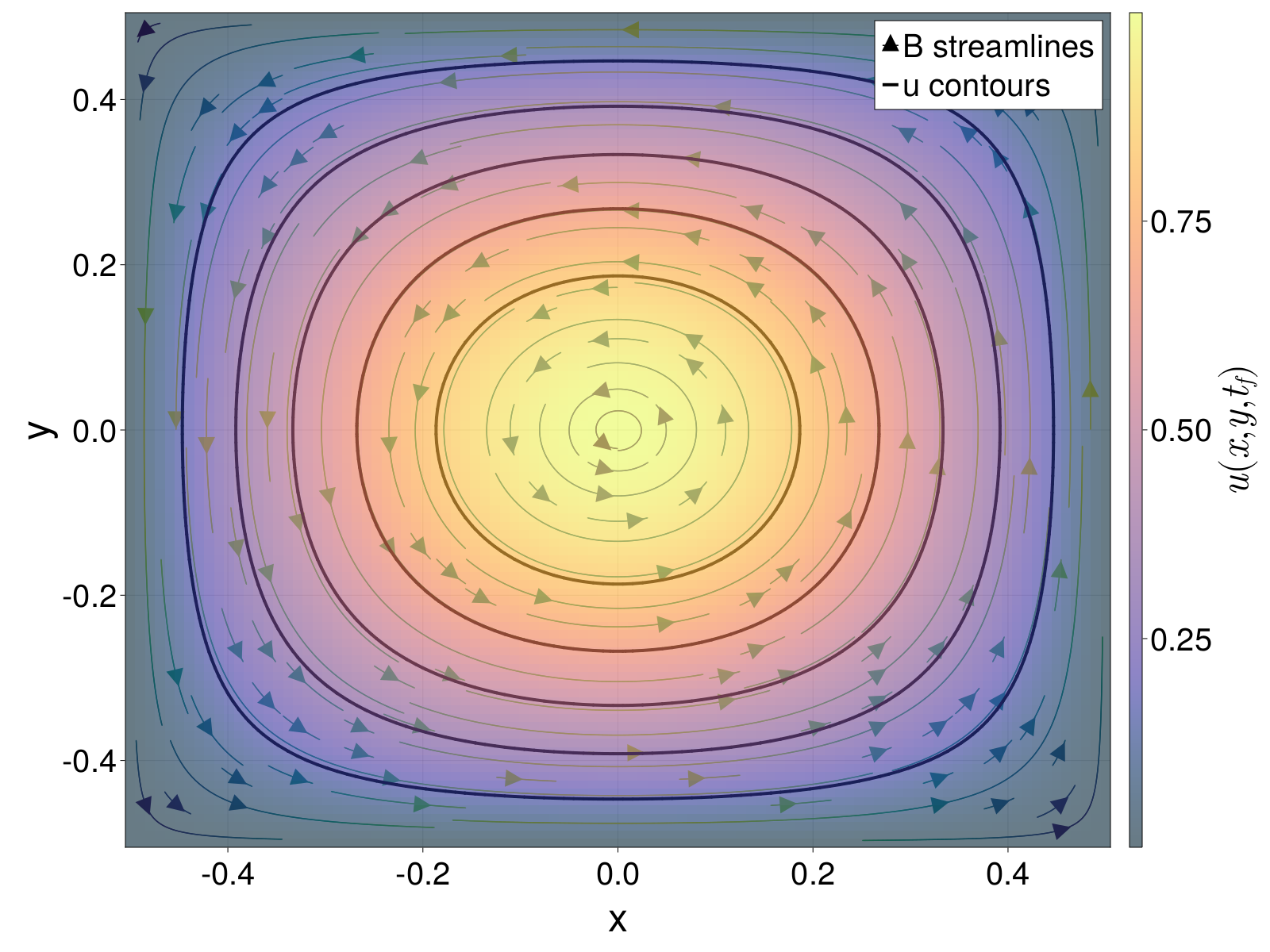}
    \caption{Exact solution of the ``NIMROD benchmark'' with $k_\perp=1$ and overlaid with contours and field lines. Field lines (arrows) match the contours of the solution (shown by the solid lines), which implies that regardless of where a field line is terminated along its trajectory it will contribute equally to the parallel diffusion.}
    \label{fig:NIMROD exact solution}
\end{figure}
As a test of the parallel penalty we use the ``NIMROD benchmark'' presented in~\cite{sovinec_nonlinear_2004}. The magnetic field is given by $\bm{B} = \bf{z}\times\nabla\psi$ where $\psi(x,y) = \cos(\pi x)\cos(\pi y)$. This gives a domain bounded by $(x,y)\in[-0.5,0.5]\cross[-0.5,0.5]$ in which all field lines are confined.
All boundary conditions for this problem are Dirichlet, $\left.u(x,y,t)\right|_{\partial\Omega}=0$.
The source term given by $F(x,y) = -\nabla^2\psi$ is also added to the right hand side in the same manner as the previous section, giving the analytic solution,
\begin{align}\label{eq:NIMROD exact}
    u(x,y,t) = \frac{1-\exp(-2t \kappa_\perp \pi^2)}{\kappa_\perp} \psi(x,y).
\end{align}
The exact solution for the case where $\kappa_\perp=1$ and magnetic field trajectories is shown in Figure \ref{fig:NIMROD exact solution}. 
Note that the choice of solution ensures that  the parallel operator is trivial and  does not contribute to diffusion since
$\bm{B}\cdot\nabla u =0$.

Since there is no toroidal ($\bm{z}$) component to the magnetic field for the NIMROD case, we re-parameterise the magnetic field $\bm{B}(\bm{x})$ to be in terms of the arclength $\bm{B}(s)$ where $s\in[0,1]$.
We note that at the origin, $\|\bm{B}\|=0$, and therefore the arclength is undefined.
This can be resolved by ensuring no grid point lies on the origin. However, we can also note that since the origin is a fixed point then the parallel map applied to this point is the identity, and it can be included in the grid with no difficulties.

\begin{remark}
    While the field lies purely in the plane, one could also consider any angle between the computational plane and the field and recover the same solution. This can be seen by considering the volumetric source term $F(x,y,z)=F(x,y)$ and any non-zero guide field $\bm{B}_0=\alpha \bm{z}$, $\alpha\in[0,1]$.
\end{remark}

Figure \ref{fig:NIMROD convergence rates} shows the spatial $l_2$ error and the  convergence rates of the error with increasing grid resolution   for a number of values of parallel diffusion coefficient as high as $\kappa_\parallel=10^{10}$. As a terminating condition for the tests we choose $t_f=0.1$ and the time step chosen to be $\Delta t= 0. 1\Delta x^2$ in order to minimise temporal errors. We show errors for both odd number of grid points, with a grid point at the origin, and even number of grid points, with no grid point at the origin, to investigate errors due to placing grid points at the origin.
Convergence rates for the second order case are around $\mathcal{O}(2)$ and $\mathcal{O}(3.5)$ for the fourth order cases. 
The  4th order method slightly outperforms the expected 3rd order global accuracy as per comments in \S\ref{sub:sec:sbp_operators}.

It is possible to achieve similar convergence rates for higher values of $\kappa_\parallel$, however this requires further tuning the penalty parameter to avoid the amplification of round off errors by the diffusion coefficient.

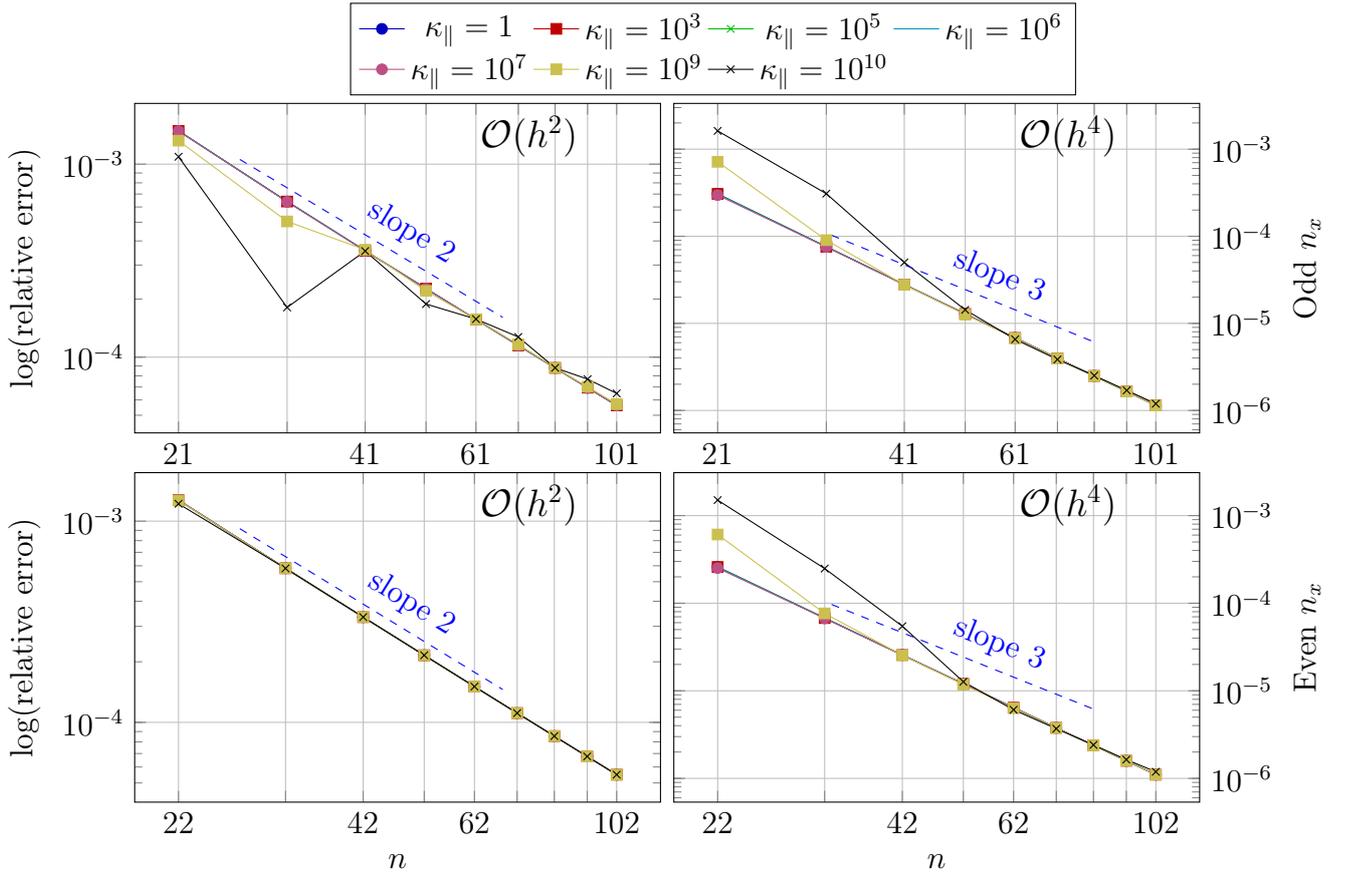
\begin{figure}[H]
    \centering
    \begin{tikzpicture}
    \begin{groupplot}[group style={
            group size=2 by 2,
            x descriptions at=edge bottom,
            vertical sep    =15pt,
            horizontal sep  =5pt
        },
        my style,
        enlargelimits=true,
        height=0.35\textwidth
    ]
    \nextgroupplot[
            NIMROD parallel legend style,
            xmode=log,ymode=log,
            xtick={21,31,41,51,61,71,81,91,101},
            xticklabels={21,,41,,61,,,,101},
            ylabel=log(relative error),
            grid]
        \addplot table[x index=0,y index=8,col sep=comma,skip first n=1]{data_NB_O2_theta0.5.csv};
        \addplot table[x index=0,y index=10,col sep=comma,skip first n=1]{data_NB_O2_theta0.5.csv};
        \addplot table[x index=0,y index=11,col sep=comma,skip first n=1]{data_NB_O2_theta0.5.csv};
        \addplot table[x index=0,y index=12,col sep=comma,skip first n=1]{data_NB_O2_theta0.5.csv};
        \addplot table[x index=0,y index=13,col sep=comma,skip first n=1]{data_NB_O2_theta0.5.csv};
        \addplot table[x index=0,y index=14,col sep=comma,skip first n=1]{data_NB_O2_theta0.5.csv};
        \addplot table[x index=0,y index=9,col sep=comma,skip first n=1]{data_NB_O2_theta0.5.csv};
        
        \logLogSlopeTriangle{0.7}{0.5}{0.83}{2}{blue,dashed};

        \node[] at (rel axis cs: 0.75,0.9) {\large $\mathcal{O}(h^2)$};

    \nextgroupplot[
        xmode=log,ymode=log,
            xtick={21,31,41,51,61,71,81,91,101},
            xticklabels={21,,41,,61,,,,101},
            ylabel=Odd $n_x$,
            yticklabel pos=right,
            grid]
        \addplot table[x index=0,y index=8,col sep=comma,skip first n=1]{data_NB_O4_theta0.5.csv};
        \addplot table[x index=0,y index=10,col sep=comma,skip first n=1]{data_NB_O4_theta0.5.csv};
        \addplot table[x index=0,y index=11,col sep=comma,skip first n=1]{data_NB_O4_theta0.5.csv};
        \addplot table[x index=0,y index=12,col sep=comma,skip first n=1]{data_NB_O4_theta0.5.csv};
        \addplot table[x index=0,y index=13,col sep=comma,skip first n=1]{data_NB_O4_theta0.5.csv};
        \addplot table[x index=0,y index=14,col sep=comma,skip first n=1]{data_NB_O4_theta0.5.csv};
        \addplot table[x index=0,y index=9,col sep=comma,skip first n=1]{data_NB_O4_theta0.5.csv};

        \logLogSlopeTriangle{0.8}{0.5}{0.6}{3}{blue,dashed};

        \node[] at (rel axis cs: 0.75,0.9) {\large $\mathcal{O}(h^4)$};

        \nextgroupplot[
            xmode=log,ymode=log,
            xtick={22,32,42,52,62,72,82,92,102},
            xticklabels={22,,42,,62,,,,102},
            ylabel=log(relative error),
            xlabel=$n$,
            grid]
        \addplot table[x index=0,y index=8,col sep=comma,skip first n=1]{data_NB_even_O2_theta0.5.csv};
        \addplot table[x index=0,y index=10,col sep=comma,skip first n=1]{data_NB_even_O2_theta0.5.csv};
        \addplot table[x index=0,y index=11,col sep=comma,skip first n=1]{data_NB_even_O2_theta0.5.csv};
        \addplot table[x index=0,y index=12,col sep=comma,skip first n=1]{data_NB_even_O2_theta0.5.csv};
        \addplot table[x index=0,y index=13,col sep=comma,skip first n=1]{data_NB_even_O2_theta0.5.csv};
        \addplot table[x index=0,y index=14,col sep=comma,skip first n=1]{data_NB_even_O2_theta0.5.csv};
        \addplot table[x index=0,y index=9,col sep=comma,skip first n=1]{data_NB_even_O2_theta0.5.csv};
        
        \logLogSlopeTriangle{0.7}{0.5}{0.83}{2}{blue,dashed};

        \node[] at (rel axis cs: 0.75,0.9) {\large $\mathcal{O}(h^2)$};

    \nextgroupplot[
        xmode=log,ymode=log,
            xtick={22,32,42,52,62,72,82,92,102},
            xticklabels={22,,42,,62,,,,102},
            ylabel=Even $n_x$,
            yticklabel pos=right,
            xlabel=$n$,
            grid]
        \addplot table[x index=0,y index=8,col sep=comma,skip first n=1]{data_NB_even_O4_theta0.5.csv};
        \addplot table[x index=0,y index=10,col sep=comma,skip first n=1]{data_NB_even_O4_theta0.5.csv};
        \addplot table[x index=0,y index=11,col sep=comma,skip first n=1]{data_NB_even_O4_theta0.5.csv};
        \addplot table[x index=0,y index=12,col sep=comma,skip first n=1]{data_NB_even_O4_theta0.5.csv};
        \addplot table[x index=0,y index=13,col sep=comma,skip first n=1]{data_NB_even_O4_theta0.5.csv};
        \addplot table[x index=0,y index=14,col sep=comma,skip first n=1]{data_NB_even_O4_theta0.5.csv};
        \addplot table[x index=0,y index=9,col sep=comma,skip first n=1]{data_NB_even_O4_theta0.5.csv};

        \logLogSlopeTriangle{0.8}{0.5}{0.6}{3}{blue,dashed};

        \node[] at (rel axis cs: 0.75,0.9) {\large $\mathcal{O}(h^4)$};

    \end{groupplot}
    \end{tikzpicture}
    
    \caption{Relative error from decreasing $\Delta x$ due to spatial discretisation. Time-step is scaled as $\Delta t=0.1\Delta x^2$. 
    Top: Odd number of grid points convergence.
    Bottom: Even number of grid points convergence.
    Left: Second order spatial discretisation with convergence rate of $\sim2$. 
    Right: Fourth order spatial discretisation with convergence rate of $\sim3.5$.}
    \label{fig:NIMROD convergence rates}
\end{figure}

Compared to previous finite difference results for the NIMROD case, \citeauthor{gunter_modelling_2005} achieved optimal convergence for their second and fourth order numerical symmetric schemes achieving errors on the order of $\sim10^{-4}$ for the fourth order scheme, at a resolution of $n\approx50$ for $\kappa_\parallel/\kappa_\perp=10^9$ \cite{gunter_modelling_2005}. \citeauthor{chacon_asymptotic-preserving_2014} obtains fourth order convergence for their fourth order numerical scheme. For the $\kappa_\parallel/\kappa_\perp=10^5$ case they achieve errors of $10^{-6}$ at $n=32$ and $10^{-8}$ for $n=128$ \cite{chacon_asymptotic-preserving_2014}.

Figure \ref{fig:NIMROD time error} shows the second order temporal convergence for varying $\kappa_\parallel$ values with $\theta=0.5$ and with fixed grid resolution and decreasing time step. Spatial resolution is fixed at $n_x=n_y=201$ to minimise spatial errors interfering with temporal convergence.

\begin{figure}[H]
    \centering
    \begin{tikzpicture}
    \begin{groupplot}[group style={
            group size=2 by 1,
            x descriptions at=edge bottom,
            horizontal sep  =5pt,
            vertical sep  =10pt,
        },
        my style,
        enlargelimits=true,
        height=0.35\textwidth
    ]
    \nextgroupplot[
            NIMROD parallel legend style,
            xmode=log,ymode=log,
            xtick=      {0.1,0.05,0.025,0.0125,0.00625,0.003125,0.0015625},
            xticklabels={0,1,2,3,4,5,6},
            xlabel=$0.1/2^i$,
            x dir=reverse,
            ylabel=log(relative error),
            grid]
        \addplot table[x index=0,y index=1,col sep=comma,skip first n=1]{data_NB_TL_n201_O2.csv};
        \addplot table[x index=0,y index=6,col sep=comma,skip first n=1]{data_NB_TL_n201_O2.csv};
        \addplot table[x index=0,y index=7,col sep=comma,skip first n=1]{data_NB_TL_n201_O2.csv};
        \addplot table[x index=0,y index=3,col sep=comma,skip first n=1]{data_NB_TL_n201_O2.csv};
        \addplot table[x index=0,y index=4,col sep=comma,skip first n=1]{data_NB_TL_n201_O2.csv};
        \addplot table[x index=0,y index=5,col sep=comma,skip first n=1]{data_NB_TL_n201_O2.csv};
        \addplot table[x index=0,y index=2,col sep=comma,skip first n=1]{data_NB_TL_n201_O2.csv};
        
        \logLogSlopeTriangle{0.7}{0.5}{0.83}{2}{blue,dashed};

        \node[] at (rel axis cs: 0.75,0.9) {\large $\mathcal{O}(h^2)$};

    \nextgroupplot[
        xmode=log,ymode=log,
            xtick=      {0.1,0.05,0.025,0.0125,0.00625,0.003125,0.0015625},
            xticklabels={0,1,2,3,4,5,6},
            xlabel=$0.1/2^i$,
            x dir=reverse,
            ylabel=Temporal error,
            yticklabel pos=right,
            grid]
        \addplot table[x index=0,y index=1,col sep=comma,skip first n=1]{data_NB_TL_n201_O4.csv};
        \addplot table[x index=0,y index=6,col sep=comma,skip first n=1]{data_NB_TL_n201_O4.csv};
        \addplot table[x index=0,y index=7,col sep=comma,skip first n=1]{data_NB_TL_n201_O4.csv};
        \addplot table[x index=0,y index=3,col sep=comma,skip first n=1]{data_NB_TL_n201_O4.csv};
        \addplot table[x index=0,y index=4,col sep=comma,skip first n=1]{data_NB_TL_n201_O4.csv};
        \addplot table[x index=0,y index=5,col sep=comma,skip first n=1]{data_NB_TL_n201_O4.csv};
        \addplot table[x index=0,y index=2,col sep=comma,skip first n=1]{data_NB_TL_n201_O4.csv};

        \logLogSlopeTriangle{0.7}{0.5}{0.83}{2}{blue,dashed};

        \node[] at (rel axis cs: 0.75,0.9) {\large $\mathcal{O}(h^4)$};

    \end{groupplot}
    \end{tikzpicture}
    
    \caption{
    Numerical error from decreasing $\Delta t$ as $0.1/2^i$ where $i\in\{0,1,\dots,6\}$ and the grid resolution is fixed at $201\times201$.
    Left: Second order spatial discretisation..
    Right: Fourth order spatial discretisation.}
    \label{fig:NIMROD time error}
\end{figure}
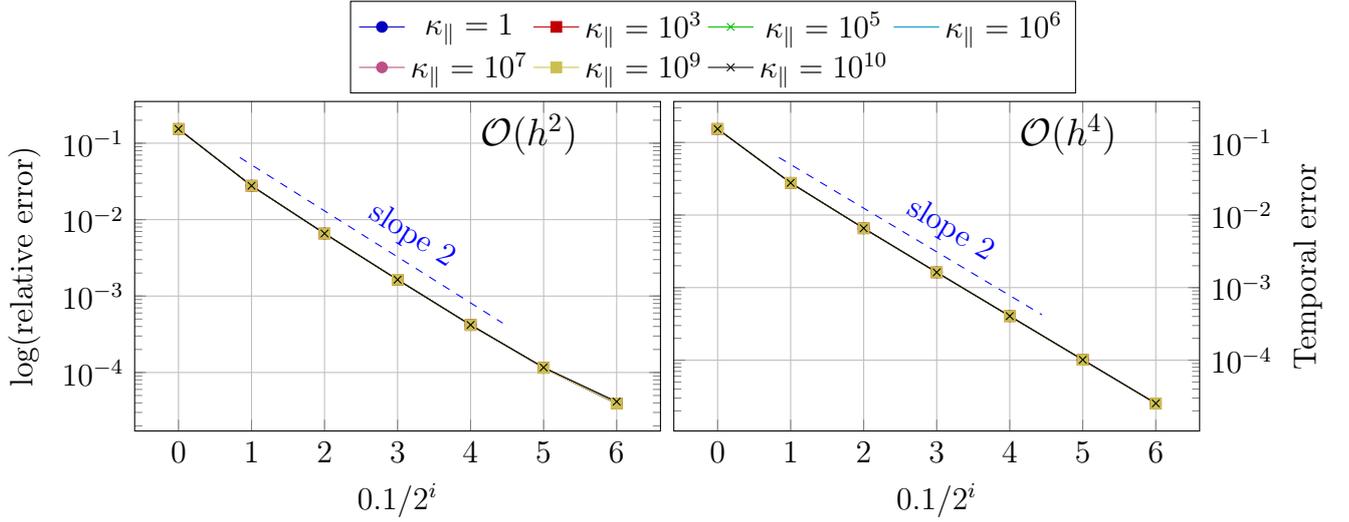

\subsubsection{Preservation of asymptotic behaviour}\label{sec:sub:Preservation of asymptotic behaviour}

In Figure \ref{fig:NIMROD pollution} we report the ``pollution'' measured as $\Delta\kappa = 1-u(0,0,t_f)$ in line with other work in this space. We see that the numerical pollution does not depend on the parallel diffusion parameter, $\kappa_\parallel$, demonstrating our method is asymptotic preserving.

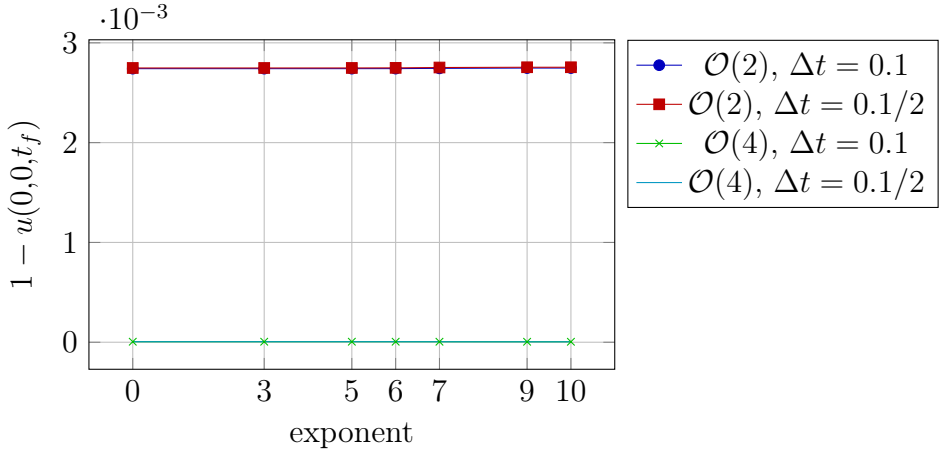
\begin{figure}[H]
    \centering
    \begin{tikzpicture}
    \begin{axis}[
            legend entries={
                $\mathcal{O}(2)\text{, }\Delta t = 0.1$,
                $\mathcal{O}(2)\text{, }\Delta t = 0.1/2$,
                $\mathcal{O}(4)\text{, }\Delta t = 0.1$,
                $\mathcal{O}(4)\text{, }\Delta t = 0.1/2$
            },
            legend style={
                at={([yshift=0pt]1.32,1.0)},
                anchor=north,
            },
            legend columns=1,
            xtick=      {0,3,5,6,7,9,10,12},
            xticklabels={0,3,5,6,7,9,10,12},
            xlabel=exponent,
            ylabel=$1-u(0\text{,}0\text{,}t_f)$,
            width=0.5\textwidth,
            height=0.35\textwidth,
            grid
            ]
        
        \addplot table[x index=0,y index=1,col sep=comma]{data_NB_AP_O2.csv};
        \addplot table[x index=0,y index=2,col sep=comma]{data_NB_AP_O2.csv};
        
        \addplot table[x index=0,y index=1,col sep=comma,skip first n=1]{data_NB_AP_O4.csv};
        \addplot table[x index=0,y index=2,col sep=comma,skip first n=1]{data_NB_AP_O4.csv};
    \end{axis}
    \end{tikzpicture}
    
    \caption{Numerical pollution, $1-u(0,0,1)$, from increasing $\kappa_\parallel$ with fixed $\Delta t=0.1/2^i$ where $i\in\{0,1\}$ and fixed order..}
    \label{fig:NIMROD pollution}
\end{figure}

\subsubsection{Investigation of parallel penalty}\label{sec:sub:NIMROD investigation of parallel penalty}

We investigate the properties of the parallel penalty in the NIMROD case. Since we expect the parallel penalty to do nothing, the penalty parameter should act to suppress the parallel operator. Figure \ref{fig:penalty during NIMROD test} shows the effective parallel penalty parameter,
\begin{align}\label{eq:effective parallel penalty}
    \tau_{\text{eff}} = \tau_\parallel \kappa_\parallel \max(\bm{H}^{-1}),
\end{align}
over the course of each simulation.
The figure demonstrates the parallel penalty is suppressed by the non-linear $\tau_\parallel$, as expected when the interpolated solution is close to the perpendicular solve. 

The slight growth in the parameter can be explained by the fact that $\bm{u}-\bm{w}$ is very close at the beginning, since the solution is close to zero. Since $u(x,y,t)$ tends to equilibrium as $t\to\infty$, the penalty parameter stabilises with $\bm{u}-\bm{w}$.

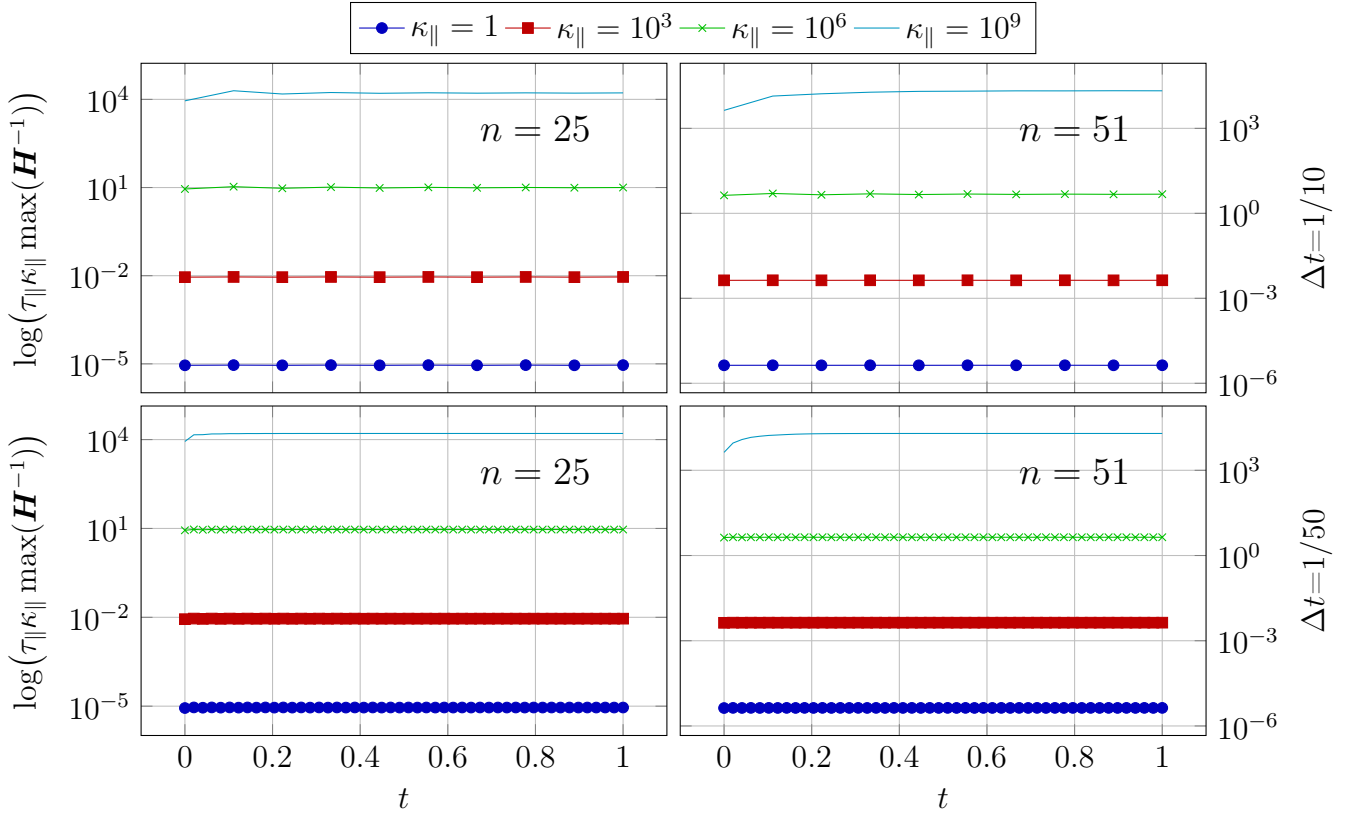
\begin{figure}[H]
    \centering
    \begin{tikzpicture}
    \begin{groupplot}[group style={
            group size=2 by 2,
            x descriptions at=edge bottom,
            horizontal sep  =5pt,
            vertical sep  =5pt,
        },
        my style,
        enlargelimits=true,
        height=0.35\textwidth
    ]
    \nextgroupplot[
            legend entries={
                $\kappa_\parallel=1$,
                $\kappa_\parallel=10^3$,
                $\kappa_\parallel=10^6$,
                $\kappa_\parallel=10^9$,
                $\kappa_\parallel=10^{12}$
            },
            legend style={
                at={([yshift=2pt]1.05,1.17)},
                anchor=north,
            },
            legend columns=5,
            ymode=log,
            ylabel=$\log(\tau_\parallel \kappa_\parallel \max(\bm{H}^{-1}))$,
            grid]
        \addplot table[x index=0,y index=1,col sep=comma]{data_NB_tau_dt0.1.csv};
        \addplot table[x index=0,y index=3,col sep=comma]{data_NB_tau_dt0.1.csv};
        \addplot table[x index=0,y index=5,col sep=comma]{data_NB_tau_dt0.1.csv};
        \addplot table[x index=0,y index=7,col sep=comma]{data_NB_tau_dt0.1.csv};
        
        \node[] at (rel axis cs: 0.75,0.8) {\large $n=25$};

    \nextgroupplot[
            ymode=log,
            ylabel=$\Delta t\text{=}1/10$,
            yticklabel pos=right,
            grid]
        \addplot table[x index=0,y index=2,col sep=comma]{data_NB_tau_dt0.1.csv};
        \addplot table[x index=0,y index=4,col sep=comma]{data_NB_tau_dt0.1.csv};
        \addplot table[x index=0,y index=6,col sep=comma]{data_NB_tau_dt0.1.csv};
        \addplot table[x index=0,y index=8,col sep=comma]{data_NB_tau_dt0.1.csv};

        \node[] at (rel axis cs: 0.75,0.8) {\large $n=51$};

    \nextgroupplot[
            ymode=log,
            xlabel=$t$,
            ylabel=$\log(\tau_\parallel \kappa_\parallel \max(\bm{H}^{-1}))$,
            grid]
        \addplot table[x index=0,y index=1,col sep=comma]{data_NB_tau_dt0.02.csv};
        \addplot table[x index=0,y index=3,col sep=comma]{data_NB_tau_dt0.02.csv};
        \addplot table[x index=0,y index=5,col sep=comma]{data_NB_tau_dt0.02.csv};
        \addplot table[x index=0,y index=7,col sep=comma]{data_NB_tau_dt0.02.csv};
        
        \node[] at (rel axis cs: 0.75,0.8) {\large $n=25$};

    \nextgroupplot[
            ymode=log,
            xlabel=$t$,
            ylabel=$\Delta t\text{=}1/50$,
            yticklabel pos=right,
            grid]
        \addplot table[x index=0,y index=2,col sep=comma]{data_NB_tau_dt0.02.csv};
        \addplot table[x index=0,y index=4,col sep=comma]{data_NB_tau_dt0.02.csv};
        \addplot table[x index=0,y index=6,col sep=comma]{data_NB_tau_dt0.02.csv};
        \addplot table[x index=0,y index=8,col sep=comma]{data_NB_tau_dt0.02.csv};

        \node[] at (rel axis cs: 0.75,0.8) {\large $n=51$};

    \end{groupplot}
    \end{tikzpicture}
    
    \caption{Value of the effective penalty parameter \eqref{eq:effective parallel penalty} during select NIMROD test using second order spatial operator. 
    Left top: grid resolution of $n_x=n_y=25$, $\Delta t=0.1$. Left right: $n_x=n_y=51$, $\Delta y=0.1$.
    Left bottom: grid resolution of $n_x=n_y=25$, $\Delta t=0.02$. Left right: $n_x=n_y=51$, $\Delta y=0.02$.
    }
    \label{fig:penalty during NIMROD test}
\end{figure}

\subsection{Single island self convergence test}

We perform a number of tests on a case in a periodic box given by $(x,y)\in[0,1]\times[0,2\pi]$ with a single magnetic island.
The flux function and magnetic field are given by,
\begin{align}
    \psi(x,y) = (x-x_1)^2 + \delta x (1-x) \cos(y), \qquad \bm{B} = \bm{z}\cross\nabla\psi + \bm{z},
\end{align}
where $\delta = 0.05$ and  $x_1=0.5$. This has a single island centred at $x_1$.
The boundary conditions in $x$ are Dirichlet, $u(0,y,t)=0$ and $u(1,y,t)=1$, and periodic in $y$. The initial condition is chosen as a ramp $u(x,y,0)=x$.
The expected behaviour of the solution is that the temperature will tend towards uniformity across the island as $\kappa_\parallel$ increases.

Simulations are run until $t_f=10^{-2}$ with a time-step of $\Delta t=10^{-4}$, chosen to ensure that we sufficiently resolve the fast diffusion and maintain temporal accuracy. The reference solution is computed on an $n_x=n_y=801$ grid using a fourth order SBP operator in the perpendicular direction.

The error is greatest near the island, so to improve accuracy grid points are packed in this region by the scaling function,
\begin{align}
    x(s,n_x) = \frac{\sinh(a \times s (n_x/51)^{b})}{2\sinh(a(n_x/51)^{b})} + 0.5, \qquad s\in[0,1],
\end{align}
where $a=0.15$ and $b=1.3$ and the $51$ corresponds to the lowest grid resolution and $n_x$ is the number of grid points at a given resolution. This parameterisation steepens the slope around the island as the grid resolution is increased.

Figure \ref{fig:single island convergence} shows the convergence of $l_2$ errors for the 2nd and 4th order spatial discretisations using both the grid packing (top) and a standard cartesian grid (bottom). The large errors shown in the cartesian case demonstrate why grid packing near the separatrix is necessary.

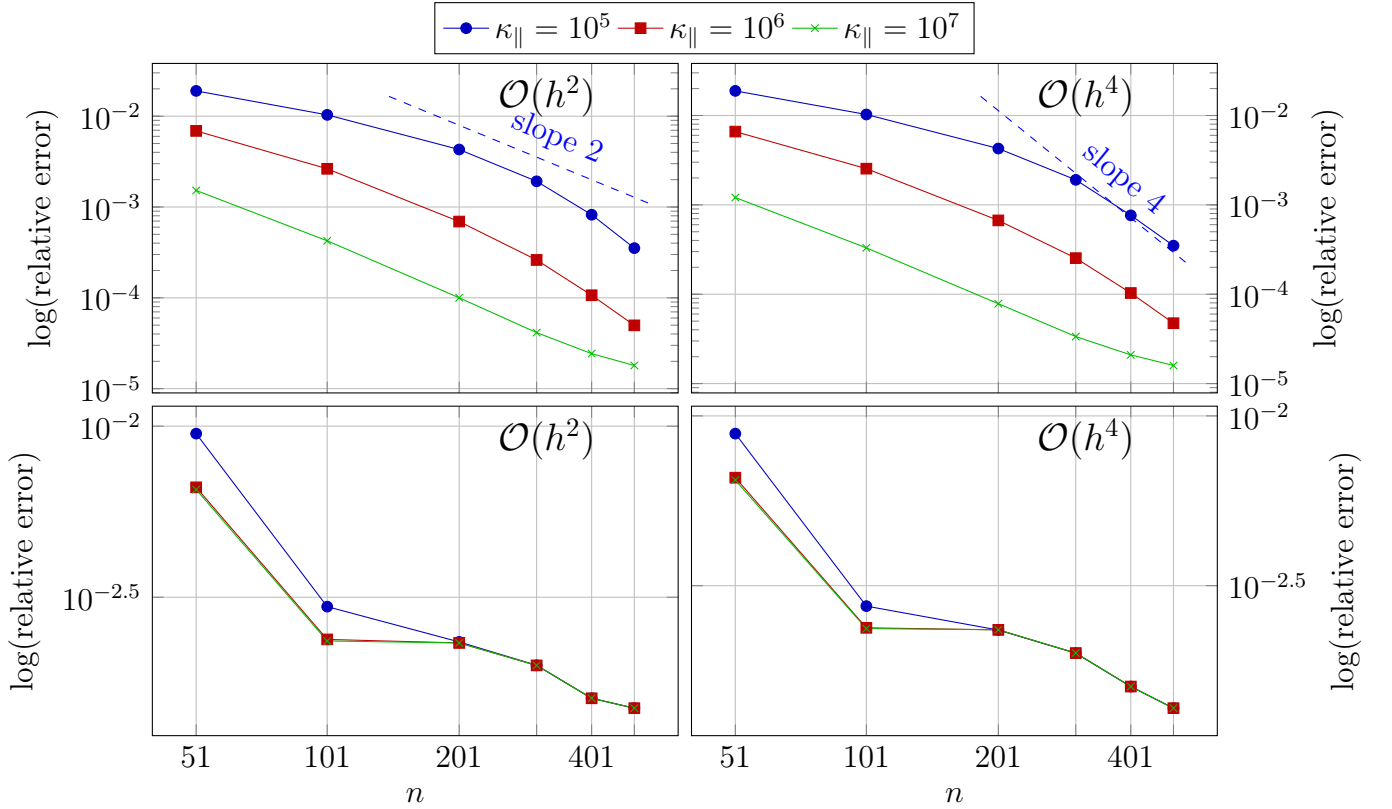
\begin{figure}[tbh]
    \centering
    \begin{tikzpicture}
        \begin{groupplot}[group style={
            group size=2 by 2,
            x descriptions at=edge bottom,
            horizontal sep  =5pt,
            vertical sep  =5pt,
        },
        my style,
        enlargelimits=true,
        height=0.35\textwidth,
        ylabel=log(relative error),
        grid
    ]

    \nextgroupplot[
        legend entries={
                $\kappa_\parallel=10^5$,
                $\kappa_\parallel=10^6$,
                $\kappa_\parallel=10^7$,
            },
        legend style={
                at={([yshift=2pt]1.05,1.17)},
                anchor=north,
            },
        legend columns=3,
            xmode=log,ymode=log,
            xtick={51,101,201,301,401,501},
            ]
        \addplot table[x index=0,y index=1,col sep=comma,skip first n=1]{data_SingleIslandField_stretch_O2_k1.0e5_delta0.05_dt4.csv};
        \addplot table[x index=0,y index=1,col sep=comma,skip first n=1]{data_SingleIslandField_stretch_O2_k1.0e6_delta0.05_dt4.csv};
        \addplot table[x index=0,y index=1,col sep=comma,skip first n=1]{data_SingleIslandField_stretch_O2_k1.0e7_delta0.05_dt4.csv};

        \logLogSlopeTriangle{0.95}{0.5}{0.9}{2}{blue,dashed};
        \node[] at (rel axis cs: 0.75,0.9) {\large $\mathcal{O}(h^2)$};

    \nextgroupplot[xmode=log,ymode=log,
            yticklabel pos=right,
            xtick={51,101,201,301,401,501},
            ]

        \addplot table[x index=0,y index=1,col sep=comma,skip first n=1]{data_SingleIslandField_stretch_O4_k1.0e5_delta0.05_dt4.csv};
        \addplot table[x index=0,y index=1,col sep=comma,skip first n=1]{data_SingleIslandField_stretch_O4_k1.0e6_delta0.05_dt4.csv};
        \addplot table[x index=0,y index=1,col sep=comma,skip first n=1]{data_SingleIslandField_stretch_O4_k1.0e7_delta0.05_dt4.csv};
        
        \logLogSlopeTriangle{0.95}{0.4}{0.9}{4}{blue,dashed};
        \node[] at (rel axis cs: 0.75,0.9) {\large $\mathcal{O}(h^4)$};

    \nextgroupplot[
            xmode=log,ymode=log,
            xtick={51,101,201,301,401,501},
            xticklabels={51,101,201,,401,}
            ]
        \addplot table[x index=0,y index=1,col sep=comma,skip first n=1]{data_SingleIslandField_cart_O2_k1.0e5_delta0.05_dt4.csv};
        \addplot table[x index=0,y index=1,col sep=comma,skip first n=1]{data_SingleIslandField_cart_O2_k1.0e6_delta0.05_dt4.csv};
        \addplot table[x index=0,y index=1,col sep=comma,skip first n=1]{data_SingleIslandField_cart_O2_k1.0e7_delta0.05_dt4.csv};

        \node[] at (rel axis cs: 0.75,0.9) {\large $\mathcal{O}(h^2)$};

    \nextgroupplot[xmode=log,ymode=log,
            yticklabel pos=right,
            xtick={51,101,201,301,401,501},
            xticklabels={51,101,201,,401,}
            ]

        \addplot table[x index=0,y index=1,col sep=comma,skip first n=1]{data_SingleIslandField_cart_O4_k1.0e5_delta0.05_dt4.csv};
        \addplot table[x index=0,y index=1,col sep=comma,skip first n=1]{data_SingleIslandField_cart_O4_k1.0e6_delta0.05_dt4.csv};
        \addplot table[x index=0,y index=1,col sep=comma,skip first n=1]{data_SingleIslandField_cart_O4_k1.0e7_delta0.05_dt4.csv};

        \node[] at (rel axis cs: 0.75,0.9) {\large $\mathcal{O}(h^4)$};

    \end{groupplot}\end{tikzpicture}
    \caption{Single island self convergence plot. Reference simulation has a resolution of $n_x=n_y=801$ all simulations have a timestep of $10^{-4}$.
    Left: Second order convergence.
    Right: Fourth order convergence.
    Top: With grid packing.
    Bottom: Without grid packing.
    }
    \label{fig:single island convergence}
\end{figure}
For the grid packing tests, we observe exponential convergence for all three test values of $\kappa_\parallel$.
Figure \ref{fig:single island surface plot of error} shows how the error changes for the $n_x=n_y=201$, $\kappa_\parallel=10^6$ simulation, demonstrating the large error about the separatrix.
\begin{figure}[H]
    \centering
    \includegraphics[width=\columnwidth]{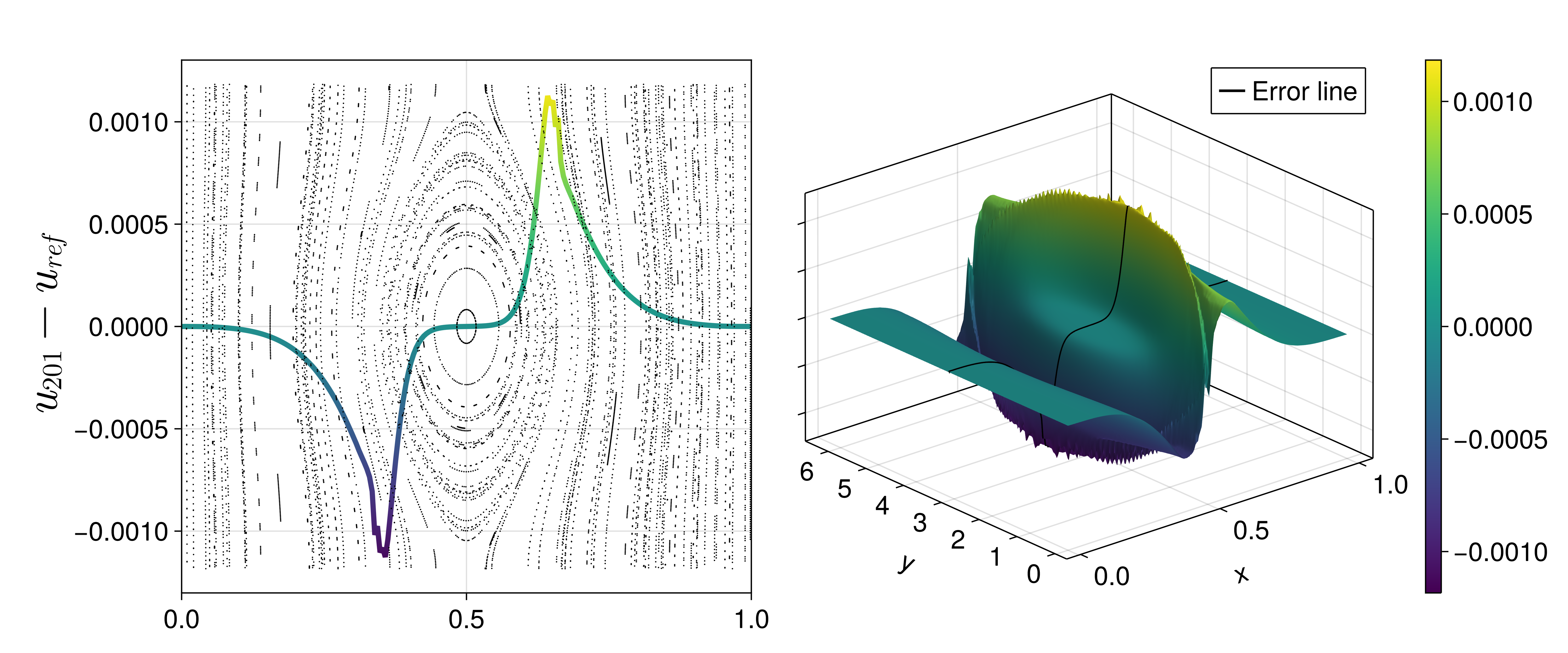}
    \caption{Error for the $n_x=n_y=201$ simulation for $\kappa_\parallel=10^6$.
    Left: The error as a function of $x$ position at $y=\pi$, corresponding to the black line in the right figure and with a Poincar\`e section showing the position of the error with respect to the field.
    Right: Error across the entire domain.}
    \label{fig:single island surface plot of error}
\end{figure}

The large errors and requirement for grid packing can be understood by looking at the error field in Figure \ref{fig:single island surface plot of error}. The solution inside and outside the island are not compatible with one another. Inside the island, the solution is constant, while outside the island the solution can be written as a linear function. Therefore for sufficiently high $\kappa_\parallel$ the solution becomes ``almost discontinuous'' and can only be properly resolved if grid points lie across the discontinuity, i.e. the separatrix.

While not being the same test (the field is different and a source term is added), the paper by \citeauthor{chacon_asymptotic-preserving_2014} also shows that convergence rates can be degraded by the presence of a separatrix \cite{chacon_asymptotic-preserving_2014}. The authors refer to work by \citeauthor{fitzpatrick_helical_1995}, which indicates that a boundary layer forms at the separatrix which is difficult to resolve with a standard grid \cite{fitzpatrick_helical_1995}.

As with the NIMROD case we show the behaviour of the effective parallel penalty, \eqref{eq:effective parallel penalty}, in Figure \ref{fig:penalty during SI test}.
We note the behaviour is significantly different to the NIMROD case, but not unexpected. In this instance the parallel penalty very rapidly enforces the difference $\bm{u} - \bm{w}$ and then relaxes as an equilibrium is achieved.
The flattening behaviour at longer time can be expected for problems with time independent Dirichlet boundary conditions and no source terms, since the solution relaxes to a steady state solution.

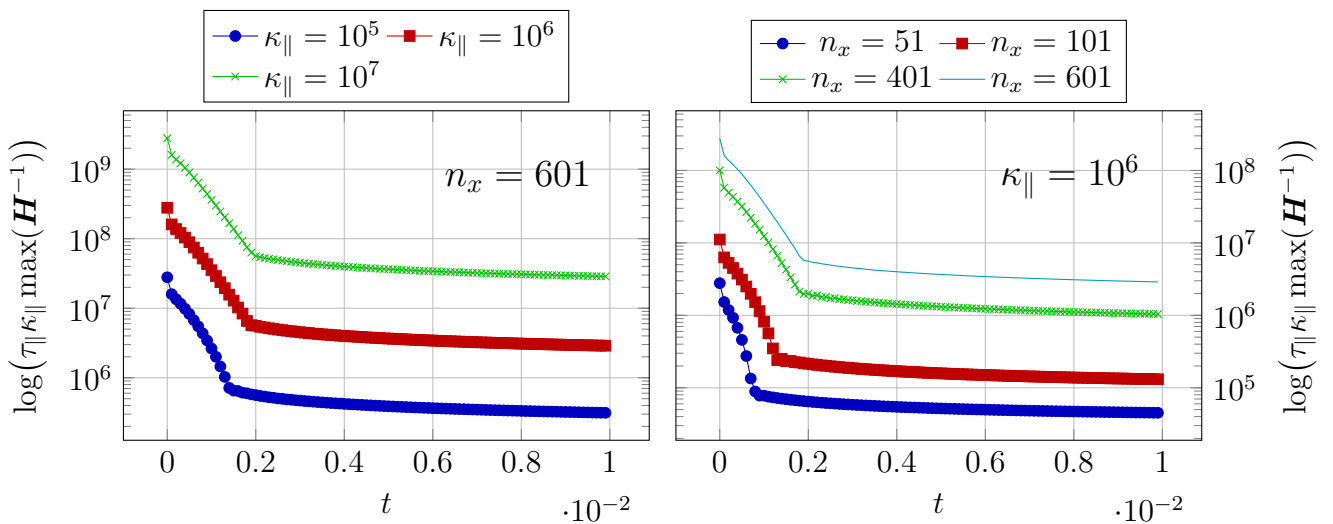
\begin{figure}[H]
    \centering
    \begin{tikzpicture}
    \begin{groupplot}[group style={
            group size=2 by 1,
            x descriptions at=edge bottom,
            horizontal sep  =10pt,
            vertical sep  =10pt,
        },
        my style,
        enlargelimits=true,
        height=0.35\textwidth
    ]
    \nextgroupplot[
            legend entries={
                $\kappa_\parallel=10^5$,
                $\kappa_\parallel=10^6$,
                $\kappa_\parallel=10^7$
            },
            legend style={
                at={([yshift=2pt]0.5,1.29)},
                anchor=north,
            },
            legend columns=2,
            ymode=log,
            xlabel=$t$,
            ylabel=$\log(\tau_\parallel \kappa_\parallel \max(\bm{H}^{-1}))$,
            grid]
        \addplot table[x index=0,y index=2,col sep=comma]{data_SingleIslandField_stretch_tau_O2_k1.0e5_delta0.05_dt4.csv};
        \addplot table[x index=0,y index=2,col sep=comma]{data_SingleIslandField_stretch_tau_O2_k1.0e6_delta0.05_dt4.csv};
        \addplot table[x index=0,y index=2,col sep=comma]{data_SingleIslandField_stretch_tau_O2_k1.0e7_delta0.05_dt4.csv};
        
        \node[] at (rel axis cs: 0.75,0.8) {\large $n_x=601$};

    \nextgroupplot[
            legend entries={
                $n_x=51$,
                $n_x=101$,
                $n_x=401$,
                $n_x=601$
            },
            legend style={
                at={([yshift=2pt]0.5,1.26)},
                anchor=north,
            },
            legend columns=2,
            ymode=log,
            xlabel=$t$,
            ylabel=$\log(\tau_\parallel \kappa_\parallel \max(\bm{H}^{-1}))$,
            yticklabel pos=right,
            grid]

        \addplot table[x index=0,y index=7,col sep=comma]{data_SingleIslandField_stretch_tau_O2_k1.0e6_delta0.05_dt4.csv};
        \addplot table[x index=0,y index=6,col sep=comma]{data_SingleIslandField_stretch_tau_O2_k1.0e6_delta0.05_dt4.csv};
        \addplot table[x index=0,y index=4,col sep=comma]{data_SingleIslandField_stretch_tau_O2_k1.0e6_delta0.05_dt4.csv};
        \addplot table[x index=0,y index=2,col sep=comma]{data_SingleIslandField_stretch_tau_O2_k1.0e6_delta0.05_dt4.csv};
        
        \node[] at (rel axis cs: 0.75,0.8) {\large $\kappa_\parallel=10^6$};

    \end{groupplot}
    \end{tikzpicture}
    
    \caption{Value of the effective penalty parameter \eqref{eq:effective parallel penalty}.
    Left: Fixed grid size $n_x=601$.
    Right: Fixed $\kappa_\parallel=10^6$.}
    \label{fig:penalty during SI test}
\end{figure}

\subsection{Periodic slab}\label{sec:sub:Periodic slab}
As mentioned in the introduction, \citeauthor{hudson_temperature_2008}~\cite{hudson_temperature_2008} showed that contours of the solution to the field aligned anisotropic diffusion equation can be used to reveal structure in the underlying magnetic field. We now seek to show qualitatively our code can extract these features also.
We use the field of \citeauthor{paul_heat_2022}~\cite{paul_heat_2022} as an example, since the field line Hamiltonian ensures $\bm{B}\cdot\bm{n}=0$ on the boundaries. First, specify the field line Hamiltonian by,
\begin{align}
    \chi(\psi,\theta,\zeta) = \psi^2/2 + \sum_m\sum_n \epsilon_{m,n}\psi(\psi-1)\cos(\theta m-\zeta n),
\end{align}
where $m,n\in\mathbb{N}$. Treating the toroidal variable ($\zeta$) as time-like, the Hamiltonian system for this choice of $\chi$ is readily given by,
\begin{align}
    \dv{\theta}{\zeta} = -\pdv{\chi}{\psi},\qquad \dv{\psi}{\zeta} = \pdv{\chi}{\zeta}.
\end{align}
Setting $m=\{2,3\}$, $n=\{1,2\}$ and $2\epsilon_{2,1}=3\epsilon_{3,2}=7.75\times10^{-3}$ generates magnetic islands at $\psi=1/2$ and $\psi=2/3$ and a chaotic region between $0.45\lesssim\psi\lesssim 0.7$. The size of the perturbation is chosen to ensure the generation of islands but is otherwise chosen arbitrarily, as the motivation here is to demonstrate the qualitative results.
We use Dirichlet conditions on the $\psi$ boundaries with $u(0,\theta,t)=0$ and $u(1,\theta,t)=1$ and periodic boundary conditions in $\theta$. The initial condition is a ramp given by $u(\psi,\theta,0)=\psi$. The diffusion coefficients have values $\kappa_\parallel=10^{10}$ and $\kappa_\perp=1$. There is no source term. To resolve the islands a high grid resolution is required. Moreover, because the island chain is centred at roughly $\psi=0.6$ we pack grid points in this region using the function
\begin{align}
    x(s) = x_1 + \alpha \sinh\left( \arcsinh\left(\frac{1-x_1}{\alpha}\right)s + \arcsinh\left(\frac{x_1}{\alpha}\right)(1-s) \right), \qquad s\in[0,1],
\end{align}
$\alpha = (100/(n_x-1))^2$, where in the test $n_x=501$.

Figure \ref{fig:PS contour} shows temperature contours at intervals of $u=0.05$ across the whole plane at a final time of $t_f=10^{-1}$ and set $\Delta t=10^{-5}$. We note the shaping of contours across the region of the island chain. Also of note is the preference for the contours to deform to the shape of the seperatrices on the main two islands. We have also plotted contours in red at values of $u(0.505,-\pi,t_f)$ and $u(0.675,y,t_f)$ which run roughly though the `O' points of the magnetic islands.
\begin{figure}[H]
    \centering
    \includegraphics[width=\textwidth]{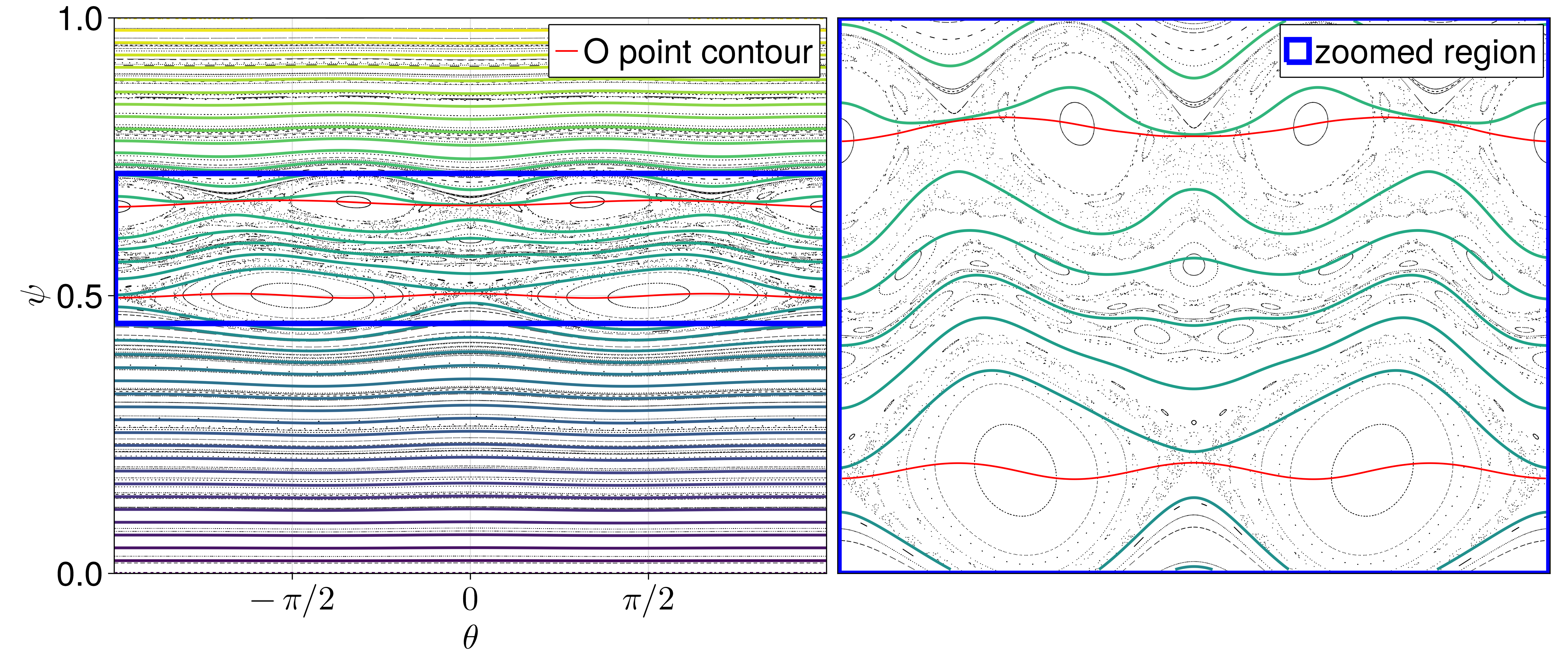}
    \caption{Coloured contours of the temperature at intervals of $u = 0.025$. Red lines lie at temperature values at the O-points of the underlying magnetic field shown as a Poincar\'e plot in black. Contours shown in red are at values of $u(0.5,\pi/2,t_f)$ and $u(0.66,\pi,t_f)$ and connect the O-points in the magnetic islands. The right figure is zoomed to the blue of the left figure to show more detail in the island region.}
    \label{fig:PS contour}
\end{figure}

The structure of the contours can be explained qualitatively as follows; 
field line transport barriers in the magnetic field dynamical system restrict the parallel transport to certain regions of the plane. In the $\psi\lesssim0.5$ and $\psi\gtrsim0.7$ region, away from the islands, the field lines roughly form a set of nested surfaces. This results in the parallel transport deforming the initially straight contours to the same shape of these surfaces.
In the island region, a mixture of the chaotic field lines and those trapped inside the islands cause the parallel transport to flatten the solution profile as it is distributed roughly evenly about the region. 
A more quantitative study relating the contours to the magnetic field structure will follow this paper.

Finally, given the previous two sections, we make a comment on choosing the grid for such problems. Grid points should be chosen to maximise the resolution near islands and chaotic regions. As can be seen in the single island case, the temperature profile far from the islands tends towards interpolating linearly from low to high solution values. The points at which the behaviour changes most rapidly is at the boundary of chaotic or island regions, such as near the separatrix. Therefore grid points should be packed near the separatrix. In the case where there may be more intricate features of the field in the chaotic region such as with the field presented in section 4.4, higher resolution may be required across the entire chaotic region and not just at its boundaries.

\section{Conclusion}\label{sec:Conclusion}

We have derived a provably stable scheme for solving the time-dependant magnetic field aligned anisotropic diffusion equation, with rigorous mathematical support.
In the perpendicular direction we approximate the spatial derivatives using Summation by Parts finite difference operators with weak implementation of boundary conditions by the SAT approach. The parallel diffusion is added by a novel penalty approach, with a nonlinear penalty parameter chosen to minimise the errors due to interpolation.

Stability is proven using the SBP formulation in the plane perpendicular to the field and a parallel operator which models the parallel Laplacian along the magnetic field lines. The stability of the semi-discrete equation is given and the time dependant scheme is shown to be unconditionally stable. 

The fully-discrete algebraic problem is solved by an operator splitting technique which accurately separates the disparate scales in the solution. The technique is shown to also be unconditionally stable.
Numerical experiments are performed to verify the theoretical analysis. The perpendicular solver is shown to converge as expected. Tests with the ``NIMROD benchmark'' \cite{sovinec_nonlinear_2004} have shown spatial and temporal convergence for the full scheme where an analytic solution is known. The benchmark also shows the method is asymptotic preserving.
A self convergence test with a single magnetic island was used to show that method converges and can reproduce the expected features of a field in the simplest case.
For the NIMROD and single island cases we have also shown the behaviour of the nonlinear penalty over the course of the simulation.

Finally we showed that the solver reproduces the features of a sample magnetic field.
The numerical methods developed here are for a Julia code which is available on GitHub at \href{https://github.com/Spiffmeister/FaADE.jl}{github.com/Spiffmeister/FaADE.jl}. Scripts which were used to generate the Figures in this paper are available at \href{https://github.com/Spiffmeister/FaADE_papera}{github.com/Spiffmeister/FaADE\_papera}. This excludes the MMS tests which are available in the main repository.

Future work will extend the method to curvilinear coordinates and investigate solutions to the anisotropic diffusion equation in magnetic fields generated by the \textit{Stepped Pressure Equilibrium Code}\footnote{\href{https://github.com/PrincetonUniversity/SPEC}{github.com/PrincetonUniversity/SPEC}}~\cite{hudson_computation_2012} or other equilibrium codes. Also of interest is investigating the relationship between the contours of the solution and the quadratic flux minimising surfaces.

\section*{Acknowledgements}

Dean Muir would like to thank the Australian Government through the Australian Government Research Training Program (RTP) Scholarship. The authors would like to thank the Simons Collaboration on Hidden Symmetries and Fusion Energy and its many participants for their discussions on the topic.

\appendix
\section{Symmetry and definiteness of the perpendicular operator}\label{sec:appendix:extra mats}

\begin{proof}\label{pf:proof of thm:stability 2d case}[Proof of theorem \ref{thm:spsd_perpendicular_diffution_operator}]
   Multiplying $\bm{P}_\perp \bm{u}$ from the left by $\bm{u}^T\bm{H}$, using the SBP properties \eqref{eq:sbp_xx}-\eqref{defn:eq:fully compatible SBP operator}  and expanding the product gives
    \begin{align*}
        \bm{u}^T\bm{H}\bm{P}_\perp\bm{u} &= 
            -\kappa_\perp\bm{u}^T \bm{D}_x^T\bm{H}\bm{D}_x\bm{u} + \kappa_\perp\bm{u}^T\bm{H}_y\bm{B}_x\bm{D}_x\bm{u} \\
        &\qquad + \tau_{x,0}\kappa_\perp\bm{u}^T(\bm{B}_x\bm{H}_x^{-1}\bm{H}_y\bm{B}_x)\bm{u} + \tau_{x,1}\kappa_\perp\bm{u}^T(\bm{D}_x^T\bm{H}_x^{-1}\bm{B}_x\bm{H}_y)\bm{u}  \\
        &\qquad -\kappa_\perp\bm{u}^T(\bm{D}_y^T\bm{H}\bm{D}_y)\bm{u} + \kappa_\perp\bm{u}^T(\bm{B}_y\bm{H}_x \bm{D}_y)\bm{u}  \\
        &\qquad + \tau_{y,0}\kappa_\perp\bm{u}^T\bm{H}_x(\bm{E}_{1y}+\bm{E}_{n_y})\bm{u} + \tau_{y,1}\kappa_\perp\bm{D}_y^T\bm{H}_x(\bm{E}_{1y}-\bm{E}_{n_y})\bm{u} \\
        &\qquad + \tau_{y,2}\kappa_\perp\bm{u}^T\bm{H}_x(\bm{E}_{1y}-\bm{E}_{n_y})\bm{D}_y\bm{u} - \kappa_\perp\bm{u}^T\bm{H}_y\bm{R}_x\bm{u} - \kappa_\perp\bm{u}\bm{H}_x\bm{R}_y\bm{u},
    \end{align*}
    where $\bm{B}_x = \bm{B}_{n_x}- \bm{B}_{1x}$ and $\bm{B}_y = \bm{B}_{n_y}- \bm{B}_{1y}$ are diagonal matrices extracting the boundary values in $x$-- and $y$--directions.
Choosing $\tau_{x,1}=-1$, $\tau_{x,0}=-(1+\tau_{x,2})$ and introducing,
    \begin{align}
        &(\bm{D}_x-\bm{H}_x^{-1}\bm{B}_x)^T \bm{H}_x (\bm{D}_x-\bm{H}_x^{-1}\bm{B}_x)= \nonumber\\
            &\qquad \bm{D}_x^T\bm{H}_x\bm{D}_x - \bm{D}_x^T\bm{B} - \bm{B}_x^T\bm{D}_x + \bm{B}_x^T\bm{H}_x^{-1}\bm{B}_x,
    \end{align}
then we have,
    \begin{align*}
        \bm{u}^T\bm{H}\bm{P}_\perp\bm{u} &= 
            -\kappa_\perp\bm{u}^T([(\bm{D}_x-\bm{H}_x^{-1}\bm{B}_x)^T \bm{H} (\bm{D}_x-\bm{H}_x^{-1}\bm{B}_x)])\bm{u}   \\
        &\qquad -\tau_{x,2}\kappa_\perp\bm{u}^T(\bm{B}_x\bm{H}^{-1}_x\bm{B}_x\bm{H}_y)\bm{u} \\
        &\qquad- \kappa_\perp\bm{u}^T(\bm{D}_y^T\bm{H}\bm{D}_y)\bm{u} + \kappa_\perp\bm{u}^T(\bm{B}_y\bm{D}_y\bm{H}_x)\bm{u} \\
        &\qquad + \tau_{y,0}\bm{u}\bm{H}_x(\bm{E}_{1y}+\bm{E}_{n_y})\bm{u} + 
            \tau_{y,1}\kappa_\perp\bm{u}^T\bm{H}_x\bm{D}_y^T(\bm{E}_{1y}-\bm{E}_{n_y})\bm{u}  \\
        &\qquad +
            \tau_{y,2}\kappa_\perp\bm{u}^T\bm{H}_x(\bm{E}_{1y}-\bm{E}_{n_y})\bm{D}_y\bm{u} -
        \bm{u}\bm{H}_{y}\bm{R}_x\bm{u} - \bm{u}\bm{H}_{x}\bm{R}_y\bm{u}.
    \end{align*}
Now choosing $\tau_{y,1}=-\tau_{y,2}=\half$ and introducing $$I_{n_x}\otimes I_{n_y} = \underbrace{I_{n_x}\otimes I_{n_y} - (\mathbf{B}_y^T\mathbf{B}_y)}_{\mathbf{I}_0} + (\mathbf{B}_y^T\mathbf{B}_y) = \mathbf{I}_0 + (\mathbf{B}_y^T\mathbf{B}_y),$$ 
so that $\mathbf{I}_0$ extracts the entries internal to the domain and $\mathbf{B}_y^T\mathbf{B}_y$ extracts the $y$ boundaries, we have
    \begin{align*}
        \bm{u}^T\bm{A}_\perp\bm{u} &= -\bm{u}^T\bm{H}\bm{P}_\perp\bm{u}=
            \bm{u}^T\kappa_\perp([(\bm{D}_x-\bm{H}_x^{-1}\bm{B}_x)^T \bm{H}_x (\bm{D}_x-\bm{H}_x^{-1}\bm{B}_x)]\bm{H}_y)\bm{u}   \\
        &+ \tau_{x,2}\kappa_\perp\bm{u}^T(\bm{B}_x\bm{H}^{-1}_x\bm{B}_x\bm{H}_y)\bm{u} +
         \kappa_\perp\bm{u}^T(\bm{D}_y^T\mathbf{I}_0\bm{H}\bm{D}_y)\bm{u} \\
        &+ \kappa_\perp\bm{u}^T(\bm{D}_y^T(\mathbf{B}_y^T\mathbf{B}_y)\bm{H}\bm{D}_y )\bm{u} - 
         \tau_{y,0}\bm{u}^T\bm{H}_x(\bm{E}_{1y} + \bm{E}_{n_y})\bm{u} \\
        &+\half\kappa_\perp\bm{u}^T\bm{D}_y^T\bm{H}_x(\bm{E}_{1y} - \bm{E}_{n_y})\bm{u} + 
         \half\kappa_\perp\bm{u}^T (\bm{E}_{1y} - \bm{E}_{n_y})\bm{H}_x\bm{D}_y\bm{u} 
        \\
        &+ \kappa_\perp\bm{u}\bm{H}_{y}\bm{R}_x\bm{u} + \kappa_\perp\bm{u}\bm{H}_{x}\bm{R}_y\bm{u}.
    \end{align*}
Collecting all the terms with $\bm{E}_{1y}$, $\bm{E}_{n_y}$  and $\mathbf{B}_y^T\mathbf{B}_y$ gives the expression,
    \begin{align*}
     \bm{u}^T\bm{A}_\perp\bm{u} &= 
        \kappa_\perp\bm{u}^T([(\bm{D}_x-\bm{H}_x^{-1}\bm{B}_x)^T \bm{H}_x (\bm{D}_x-\bm{H}_x^{-1}\bm{B}_x)]\bm{H}_y)\bm{u}  \\
        &+ \tau_{x,2}\kappa_\perp\bm{u}^T(\bm{B}_x\bm{H}^{-1}_x\bm{B}_x\bm{H}_y)\bm{u}\\
        &+ \kappa_\perp\bm{u}^T(\bm{D}_y^T\bm{H}\bm{I}_0\bm{D}_y)\bm{u} + 
        \sum_{i=1}^{n_x} \Delta x h_i \bm{v}_i^T \mathcal{A} \bm{v}_i \\
            &+\bm{u}\left(\kappa_\perp\bm{H}_{y}\bm{R}_x + \kappa_\perp\bm{H}_{x}\bm{R}_y\right)\bm{u},\numberthis\label{eq:AP:SPD last}
    \end{align*}
where
    \begin{align}
        \mathcal{A} = \begin{bmatrix}
                -\tau_{y,0}          & \half \kappa_\perp            & \half \kappa_\perp \\
                \half \kappa_\perp  & \kappa_\perp \Delta y h_1     & 0 \\
                \half \kappa_\perp  & 0                             & \kappa_\perp \Delta y h_{n_y}
        \end{bmatrix}, \qquad 
        \bm{v}_i = \begin{bmatrix}
            (\bm{u})_{i,1} - (\bm{u})_{i,n_y} \\ (D_y \bm{u})_{i,1} \\ (D_y u)_{i,n_y}
        \end{bmatrix}.
    \end{align}

The matrix $\mathcal{A}$ is symmetric positive semi-definite if $\tau_{y,0} \leq -\frac{\kappa_\perp}{2\Delta y}\max\left(\frac{1}{h_{n_y}},\frac{1}{h_1}\right)$.
Thus if $\tau_{x,2} \ge 0$ then the matrix $\bm{A}_\perp$ defined by equation \eqref{eq:AP:SPD last} is symmetric positive semi-definite, that is
    \begin{align}
        \bm{u}^T\bm{A}_\perp\bm{u} &\geq 0.
    \end{align}
This completes the proof.
\end{proof}

\section{Conjugate Gradient Algorithm}\label{sec:conjugate_gradient}

The conjugate gradient algorithm used for the perpendicular solve is given below in  Algorithm \ref{alg:cg}. Note that the conjugate gradient solve converges in the norm defined by $\bm{H}$, and \textit{not} in the standard $l^2$ inner product. This is because the perpendicular operator $\bm{P}_\perp$ on line \ref{alg:CG:D} in  Algorithm \ref{alg:cg} below is symmetric negative semi-definite under the scalar product defined by $\bm{H}$, but not in $l^2$.

\SetKwComment{tjl}{\#}{}
\begin{algorithm}[H]
\DontPrintSemicolon
\caption{Conjugate Gradient solve with $\bm{H}$-norm}
\label{alg:cg}
\KwData{right hand side $\bm{b} = \bm{u}^l + (1-\theta)\Delta t \bm{P}_\perp\bm{u}^l + \bar{\bm{F}} + \operatorname{SAT}_D$ where $\operatorname{SAT}_D$ is the boundary data and $\operatorname{SAT}_S$ is the solution component of the SAT,\\
initial guess of solution $\bm{u}^l$}
\KwResult{$\bm{u}^{l+\half}$}

Let $\bm{P}_\perp$ be the perpendicular operator from equation \eqref{eq:Perpendicular Operator} with $\bm{g}=0$. \label{alg:CG:D}\;
Define $\|\bm{u}\|_{\bm{H}}:=\sqrt{\bm{u}^T\bm{H}\bm{u}} \leftarrow \sqrt{\sum_{i,j} u_{ij} h_ih_j u_{ij}\Delta{x}\Delta{y}}$\;
\Begin{
$k = 0$\;
$\bm{u}_k = \bm{u}^{l}$\;
$\bm{r}_k = (\bm{u}_k - \theta\Delta t \bm{P}_\perp\bm{u}_k) - \bm{b}$\;
$\bm{d}_k = -\bm{r}_k$

    \While{$\|\bm{r}_k\|_{\bm{H}}\leq rtol\|\bm{u}_k\|_{\bm{H}}$ {\bf and} $k<MAXIT$}{
        $d_kAd_k = \bm{d}_k^T\bm{H}(\bm{d}_k - \Delta t \theta \bm{P}_\perp\bm{d}_k) = \bm{d}_k^T(\bm{H}+ \Delta t \bm{A}_\perp)\bm{d}_k >0$\;
        $\alpha_k = -\bm{r}_k^T\bm{H}\bm{d}_k/d_kAd_k$\;
        $\bm{u}_{k+1} = \bm{u}_k + \alpha_k\bm{d}_k$\;
        $\bm{r}_k = \bm{u}_{k+1} - \Delta t \theta \bm{P}_\perp\bm{u}_{k+1} - \bm{b}$\;
        $\beta_k = r_kAr_k/d_kAd_k$\;
        $\bm{d}_k = -\bm{r}_k + \beta_k \bm{d}_k$\;
    
        $k = k + 1$\;
}}
\end{algorithm}

\section{Coefficients used in the SBP operators}\label{sec:appendix:coeffs sbp operators}

Here we report the coefficients of the second order stencils for the SBP operators using in this work. The coefficients for higher order stencils become increasingly complex, and so for the coefficients for the higher order stencils we refer readers to GitHub repository for this project \href{https://github.com/Spiffmeister/FaADE.jl}{github.com/Spiffmeister/FaADE.jl}, or to the appendix of \cite{mattsson_summation_2012}.

The $H$ operator in the second order case is given by,
\begin{align}
    H = \operatorname{diag}(1/2, 1, \dots, 1, 1/2).
\end{align}

A $2p$th order SBP scheme has order $p$ on the boundaries. The internal nodes of the 1D first derivative SBP operator 
\begin{align}
    D_x = \frac{1}{\Delta x}\begin{bmatrix}
-1 & 1 &  &  &  \\
-1 & 0 & 1 &  &  \\
 &  & \ddots &  &  \\
 &  & -1 & 0 & 1 \\
 &  &  & -1 & 1 
\end{bmatrix}
\end{align}

Recall the derivative term in the variable coefficient second derivative SBP operator is given by $M^{(k)}=D_x^THK_xD_x$. This gives the the internal nodes,
\begin{align}
    M^{(k)}_{i,i-1} &= -\frac{1}{2\Delta x^2} ([K_x]_i + [K_x]_{i-1}), \\
    M^{(k)}_{i,i} &= \frac{1}{2\Delta x^2} ([K_x]_{i+1} + [K_x]_i + [K_x]_{i-1}), \\
    M^{(k)}_{i,i+1} &= -\frac{1}{2\Delta x^2} ([K_x]_{i+1} + [K_x]_i).
\end{align}
The boundary nodes are given by,
\begin{align}
    [D_{xx}^{(k)}]_1 = [D_{xx}^{(k)}]_N = 0.
\end{align}
We refer readers to appendix A of \cite{mattsson_summation_2012} for higher order operators.

\section{Proof of theorem \ref{theo:well-posedness}, well-posedness of field aligned anisotropic diffusion equation}
Since we are interested in the energy estimate for our numerical scheme we reproduce the proof of theorem \ref{theo:well-posedness} below.
\begin{proof}
    We set $F = 0$ and multiply \eqref{eq:ADE field aligned split} by $u$ and integrate over the spatial domain $\Omega$, we have
    \begin{align}
        \int_\Omega u\dv{u}{t} \;\dd x\dd y     &= \int_\Omega u\nabla\cdot(\kappa\nabla u) \;\dd x\dd y + \int_\Omega u \mathcal{P}_\parallel u \;\dd x\dd y.
    \end{align}
    Integration by parts gives
    \begin{align}
        \frac{1}{2}\dv{t}\|u\|^2   &= -\int_\Omega \kappa_\perp \nabla u\cdot \nabla u\;\dd x\dd y   + \int_\Omega u \mathcal{P}_\parallel u \;\dd x\dd y \\
        \nonumber
        &+ \int_{y_L}^{y_R}\left.u\kappa_\perp\pdv{u}{x}\right|_{x_L}^{x_R} \dd y + \int_{x_L}^{x_R}\left.u\kappa_\perp\pdv{u}{y}\right|_{y_L}^{y_R} \dd x.
    \end{align}
    Imposing the periodic boundary condition \eqref{eq:Boundary y periodic} and the Dirichlet boundary condition \eqref{eq:Boundary x Dirichlet} with homogeneous data $g_L(y,t)=g_R(y,t)=0$ eliminates the boundary terms. Adding the conjugate transpose gives,
    \begin{align*}%
        \dv{t}\|u\|^2 &= -2\int_\Omega \nabla u\cdot \kappa_\perp\nabla u\;\dd x\dd y + \int_\Omega u(\mathcal{P}_\parallel + \mathcal{P}_\parallel^\dagger) u \; \dd x\dd y \le 0.
    \end{align*}
    Thus we have
     \begin{align*}
        \dv{t}\|u\|^2 \leq 0 \implies \|u\| \le \|f\|.
     \end{align*}
    This completes the proof.
\end{proof}

\printbibliography

\end{document}